\documentclass[preprint]{elsarticle}

\usepackage[latin1]{inputenc}
\usepackage{mathptmx}
\usepackage{amssymb}
\usepackage{amsmath}
\usepackage{hyperref}

\usepackage{pgf, tikz}

\definecolor{veryverylight}{gray}{0.85}
\definecolor{verylight}{gray}{0.75}
\definecolor{light}{gray}{0.65}
\definecolor{medium}{gray}{0.55}

\usepackage{algorithm} 
\usepackage[noend]{algpseudocode}
\algblock[parallel for]{ParFor}{EndParFor}
\algnewcommand\algorithmicparfor{\textbf{parallel\ for}}
\algnewcommand\algorithmicpardo{\textbf{do}}
\algnewcommand\algorithmicendparfor{\textbf{end\ parallel\ for}}
\algrenewtext{ParFor}[1]{\algorithmicparfor\ #1\ \algorithmicpardo}
\algrenewtext{EndParFor}{\algorithmicendparfor}
\makeatletter
\ifthenelse{\equal{\ALG@noend}{t}}%
  {\algtext*{EndParFor}}
  {}%
\makeatother

\newtheorem{theorem}{Theorem}
\newtheorem{lemma}[theorem]{Lemma}

\newtheorem{proposition}[theorem]{Proposition}

\newdefinition{remark}[theorem]{Remark}
\newdefinition{definition}[theorem]{Definition}

\newproof{proof}{Proof}

\def\ZZ{{\mathbb Z}}

\def\RR{{\mathbb R}}
\def\QQ{{\mathbb Q}}
\def\CC{{\mathbb C}}

\def\cH{{\mathcal H}}

\def\cG{{\mathcal G}}
\def\cK{{\mathcal K}}

\def\Hilb{\operatorname{Hilb}}

\def\GL{\operatorname{GL}}

\def \cone {\operatorname{cone}}
\def\conv{\operatorname{conv}}

\def\hht{\operatorname{ht}}
\def\vol{\operatorname{vol}}
\def\para{\operatorname{par}}
\def\tr{\textup{tr}}

\def\ttt#1{\texttt{#1}}

\let\epsilon=\varepsilon

\def\strut{\rule{0ex}{2.5ex}}

\journal{J. Symb. Comp.}

\begin{document}

\begin{frontmatter}

\title{The power of pyramid decomposition in Normaliz}
\author[uos]{Winfried Bruns}
\ead{wbruns@uos.de}

\author[ssi]{Bogdan Ichim}
\ead{bogdan.ichim@imar.ro}

\author[uos]{Christof S\"{o}ger}
\ead{csoeger@uos.de}

\address[uos]{Universit\"{a}t Osnabr\"{u}ck, FB Mathematik/Informatik\\ 49069 Osna\-br\"{u}ck, Germany}
\address[ssi]{Simion Stoilow Institute of Mathematics of the Romanian Academy, Research Unit 5\\ C.P. 1-764, 010702 Bucharest, Romania}

\begin{abstract}
We describe the use of pyramid decomposition in Normaliz, a
software tool for the computation of Hilbert bases and
enumerative data of rational cones and affine mono\-ids. Pyramid
decomposition in connection with efficient parallelization and
streamlined evaluation of simplicial cones has enabled Normaliz
to process triangulations of size $\approx 5\cdot 10^{11}$ that
arise in the computation of Ehrhart series related to the theory of
social choice.
\end{abstract}


\end{frontmatter}

\section{Introduction}

Normaliz is a software tool for the computation of Hilbert
bases and enumerative data of rational cones and affine
monoids. In the 17 years of its existence it has found numerous
applications; for example, in integer programming (Bogart, Raymond and Thomas
\cite{BRT}), algebraic geometry (Craw, Maclagan and Thomas \cite{CMT05}),
theoretical physics (Kappl, Ratz and Staudt \cite{KRS}), commutative algebra
(Sturmfels and Welker \cite{SW}) or elimination theory (Emiris, Kalinka,
Konaxis and Ba \cite{EKKB}). Normaliz is used in polymake \cite{JMP}, a
computer system for polyhedral geometry, and in Regina \cite{Reg}, a system for
computations with $3$-manifolds.

The mathematics of the very first version was described in
Bruns and Koch \cite{BK}, and the details of version 2.2 (2009)
are contained in Bruns and Ichim \cite{BI}. In this article we
document the mathematical ideas and the most recent
development
\footnote{Version 3.0 is available from \url{http://www.math.uos.de/normaliz}.}
resulting from them. It has extended the scope of Normaliz by
several orders of magnitude.

In algebraic geometry the spectra of algebras $K[C\cap L]$
where $C$ is a pointed cone and $L$ a lattice, both contained in
a space $\RR^d$, are the building blocks of toric varieties;
for example, see Cox, Little and Schenck \cite{CLS}. In
commutative algebra the algebras $K[C\cap L]$ which are exactly
the normal affine monoid algebras are of interest themselves.
It is clear that an algorithmic approach to toric geometry or
affine monoid algebras depends crucially on an efficient
computation of the unique minimal system of generators of a
monoid $C\cap L$ that we call its \emph{Hilbert basis}. Affine
monoids of this type are extensively discussed by Bruns and
Gubeladze \cite{BG}. The existence and uniqueness of such a
minimal system of generators is essentially due to Gordan
\cite{Go} and was proven in full generality by van der Corput
\cite{vdC}.

The computation of Hilbert bases amounts to solving homogeneous linear diophantine systems of inequalities (defining the cone) and equations and congruences (defining the lattice). Since version 2.11 Normaliz also solves inhomogeneous linear diophantine systems; in other words, it computes lattice points in polyhedra (and not just cones).

The term ``Hilbert basis'' was actually coined in integer
programming (with $L=\ZZ^d$) by Giles and Pulleyblank \cite{GP}
in connection with totally dual integral (TDI) systems. Also
see Schrijver \cite[Sections 16.4 and 22.3]{Scr}. One should
note that in integer programming usually an arbitrary, not
necessarily minimal, system of generators of $C\cap \ZZ^d$ is
called a Hilbert basis of $C$. From the computational viewpoint
and also in bounds for such systems of generators, minimality
is so important that we include it in the definition. Aardal,
Weismantel and Wolsey \cite{AWW} discuss Hilbert bases and
their connection with Graver Bases (of sublattices) and Gröbner
bases (of binomial ideals). (At present, Normaliz does not
include Graver or Gröbner bases; 4ti2 \cite{4ti2} is a tool for their
computation.) It should
be noted that Normaliz, or rather a predecessor, was
instrumental in finding a counterexample to the Integral
Carathéodory Property (Bruns, Gubeladze, Henk, Weismantel and
Martin \cite{BGHMW}) that was proposed by Seb\H{o} \cite{Seb}.
For more recent developments in nonlinear optimization using
Graver bases, and therefore Hilbert bases, see J. De Loera, R.
Hemmecke, S. Onn, U.G. Rothblum, R. Weismantel \cite{DHORW},
Hemmecke, Köppe and Weismantel \cite{HKW}, and Hemmecke, Onn
and Weismantel \cite{HOW}.

Hilbert functions and polynomials of graded algebras and
modules were introduced by Hilbert himself \cite{Hi} (in
contrast to Hilbert bases). These invariants, and the
corresponding generating functions, the Hilbert series, are
fundamental in algebraic geometry and commutative algebra. See
\cite[Chapter 6]{BG} for a brief introduction to this
fascinating area. Ehrhart functions were defined by Ehrhart
\cite{Ehr} as lattice point counting functions in multiples of
rational polytopes; see Beck and Robbins \cite{BeRo} for a
gentle introduction. Stanley \cite{Sta1} interpreted Ehrhart
functions as Hilbert functions, creating a powerful link
between discrete convex geometry and commutative algebra. In
the last decades Hilbert functions have been the objective of a
large number of articles. They even come up in optimization
problems; for example, see De Loera, Hemmecke, Köppe and
Weismantel \cite{DHKW}. Surprisingly, Ehrhart functions have an
application in compiler optimization; see Clauss, Loechner and Wilde
\cite{Clauss} for more information.

From the very beginning Normaliz has used lexicographic triangulations;
see \cite{BI}, \cite{BK}   for the use in Normaliz and De Loera,
Rambau and Santos \cite{DRS} for (regular) triangulations of polytopes.
(Since version 2.1 Normaliz contains a second,
triangulation free Hilbert basis algorithm, originally due to
Pottier \cite{Pot} and called \emph{dual} in the following; see
\cite{BI}). Lexicographic triangulations are essentially
characterized by being incremental in the following sense.
Suppose that the cone $C$ is generated by vectors
$x_1,\dots,x_n\in\RR^d$; set $C_0=0$ and $C_i=\RR_+x_1+\dots+\RR_+x_i$,
$i=1,\dots,n$. Then the lexicographic triangulation $\Lambda$
(for the ordered system $x_1,\dots,x_n$) restricts to a
triangulation of $C_i$ for $i=0,\dots,n$. Lexicographic
triangulations are easy to compute, and go very well with
Fourier-Motzkin elimination that computes the support
hyperplanes of $C$ by successive extension from $C_i$ to
$C_{i+1}$, $i=0,\dots,n-1$. The triangulation $\Lambda_{i}$ of
$C_{i}$ is extended to $C_{i+1}$ by all simplicial cones
$F+\RR_+x_{i+1}$ where $F\in \Lambda_i$ is visible from
$x_{i+1}$.

As simple as the computation of the lexicographic triangulation
is, the algorithm in the naive form just described has two
related drawbacks: (i)~one must store $\Lambda_i$ and this
becomes very difficult for sizes $\ge 10^8$; (ii)~in order to
find the facets $F$ that are visible from $x_{i+1}$ we must
match the simplicial cones in $\Lambda_i$ with the support
hyperplanes of $C_i$ that are visible from $ x_{i+1}$. While
(i)~is a pure memory problem, (ii)~quickly leads to impossible
computation times.

\emph{Pyramid decomposition} is the basic idea that has enabled
Normaliz to compute dimension $24$ triangulations of size
$\approx 5\cdot 10^{11}$ in acceptable time on standard
multiprocessor systems such as SUN xFire 4450 or Dell PowerEdge
R910. Instead of going for the lexicographic triangulation
directly, we first decompose $C$ into the pyramids generated by
$x_{i+1}$ and the facets of $C_i$ that are visible from
$x_{i+1}$, $i=0,\dots,n-1$. These pyramids (of level $0$) are
then decomposed into pyramids of level $1$ etc. While the level
$0$ decomposition need not be a polyhedral subdivision in the
strict sense, pyramid decomposition stops after finitely many
iterations at the lexicographic triangulation; see Section
\ref{LexTri} for the details and Figure \ref{FigPyrDec} for a simple
example.

Pure pyramid decomposition is very memory friendly, but
its computation times are even more forbidding than those of
pure lexicographic triangulation since too many Four\-ier-Motzkin
eliminations become necessary, and almost all of them are
inevitably wasted. That Normaliz can nevertheless cope with
extremely large triangulations relies on a well balanced
combination of both strategies that we outline in Section
\ref{Curr}.

It is an important aspect of pyramid decomposition that it is
very parallelization friendly since the pyramids can be treated
independently of each other. Normaliz uses OpenMP for shared
memory systems. Needless to say that triangulations of the size
mentioned above can hardly be reached in serial computation.

For Hilbert basis computations pyramid decomposition has a
further and sometimes tremendous advantage: one can avoid the
triangulation of those pyramids for which it is a priori clear
that they will not supply new candidates for the Hilbert basis.
This observation, on which the contribution of the authors to
\cite{BHIKS} (jointly with Hemmecke and Köppe) is based,
triggered the use of pyramid decomposition as a general
principle. See Remark \ref{partial} for a brief discussion.

In Section \ref{Eval} we describe the steps by which Normaliz
evaluates the simplicial cones in the triangulation for the
computation of Hilbert bases, volumes and Hilbert series. After the introduction of pyramid decomposition,
evaluation almost always takes significantly more time than the
triangulation. Therefore it must be streamlined as much as
possible. For the Hilbert series Normaliz uses a Stanley
decomposition \cite{Sta}. That it can be found efficiently
relies crucially on an idea of Köppe and Verdoolaege \cite{KV}.

We document the scope of Normaliz's computations in Section
\ref{Comp}. The computation times are compared with those of
4ti2 \cite{4ti2} (Hilbert bases) and LattE \cite{LatInt}
(Hilbert series). The test examples have been chosen from the
literature (Beck and Ho\c{s}ten \cite{BeH}, Ohsugi and Hibi
\cite{OH}, Schürmann \cite{Sch}, Sturmfels and Welker
\cite{SW}), the LattE distribution and the Normaliz
distribution. The desire to master the Hilbert series
computations asked for in Schürmann's paper \cite{Sch} was an
important stimulus in the recent development of Normaliz.

\section{Overview of the Normaliz algorithm}\label{Over}

The \emph{primal} Normaliz algorithm is triangulation based, as mentioned in the introduction. Normaliz contains a second, \emph{dual} algorithm for the computation of Hilbert bases that implements ideas of Pottier \cite{Pot}. The dual algorithm is treated in \cite{BI}, and has not changed much in the last years. We skip it in this article, except in Section \ref{Comp} where computation times of the primal and dual algorithm will be compared.

The primal algorithm starts from a pointed rational cone $C\subset\RR^d$ given by a system of generators $x_1,\dots,x_n$ and a sublattice $L\subset\ZZ^d$ that contains $x_1,\dots,x_n$. (Other types of input data are first transformed into this format.) The algorithm is composed as follows:
\begin{enumerate}
\item Initial coordinate transformation to $E=L\cap (\RR x_1+\dots+\RR x_n)$;
\item Fourier-Motzkin elimination computing the support hyperplanes of $C$;
\item pyramid decomposition and computation of the lexicographic triangulation $\Delta$;
\item evaluation of the simplicial cones in the triangulation:
\begin{enumerate}
\item enumeration of the set of lattice points $E_\sigma$ in the fundamental domain of a simplicial subcone $\sigma$,
\item reduction of $E_\sigma$ to the Hilbert basis $\Hilb(\sigma)$,
\item Stanley decomposition for the Hilbert series of $\sigma\cap L$;
\end{enumerate}
\item Collection of the local data:
\begin{enumerate}
\item reduction of $\bigcup_{\sigma\in\Delta} \Hilb(\sigma)$ to $\Hilb(C\cap L)$,
\item accumulation of the Hilbert series of the $\sigma\cap L$;
\end{enumerate}
\item reverse coordinate transformation to $\ZZ^d$.
\end{enumerate}

The algorithm does not strictly follow this chronological order, but interleaves steps 2--5 in an intricate way to ensure low memory usage and efficient parallelization. The steps 2 and 5 are treated in \cite{BI}, and there is not much to add here, except that 2 is now modified by the pyramid decomposition. Step 3 is described in Sections \ref{LexTri} and \ref{Curr}, and step 4 is the subject of Section \ref{Eval}. In view of the initial and final coordinate transformation we can assume $E=\ZZ^d$, and suppress the reference to the lattice in the following.

Note that the computation goals of Normaliz can be restricted, for example to the volume of a rational polytope. Then the evaluation of a simplicial cone just amounts to a determinant calculation. Another typical restricted computation goal is the lattice points contained in such a polytope. Then the reduction is replaced by a selection of degree $1$ points from the candidate set.

The algorithms described in this paper have been implemented in version 3.0.

\section{Lexicographic triangulation and pyramid decomposition}\label{LexTri}

\subsection{Lexicographic triangulation}

Consider vectors $x_1,\dots,x_n\in\RR^d$. For Normaliz these
must be integral vectors, but integrality is irrelevant in this
section. We want to compute the support hyperplanes of the cone
$$
C=\cone(x_1,\dots,x_n)=\RR_+x_1+\dots+\RR_+x_n
$$
and a triangulation of $C$  with rays through $x_1,\dots,x_n$. Such a triangulation is a
polyhedral subdivision of $C$ into simplicial subcones $\sigma$
generated by linearly independent subsets of
$\{x_1,\dots,x_n\}$.

For a triangulation $\Sigma$ of a cone $C$ and a subcone $C'$
we set
$$
\Sigma|C'=\{\sigma\cap C': \sigma\in \Sigma\}.
$$
In general $\Sigma|C'$ need not be a triangulation of $C'$, but
it is so if $C'$ is a face of $C$.

The \emph{lexicographic} (or \emph{placing}) triangulation
$\Lambda(x_1,\dots,x_n)$ of $\cone(x_1,\dots,x_n)$ can be
defined recursively as follows: (i)~the triangulation of the
zero cone is the trivial one, (ii)~$\Lambda(x_1,\dots,x_n)$ is
given by
$$
\Lambda(x_1,\dots,x_n)= \Lambda(x_1,\dots,x_{n-1})\cup \{\cone(\sigma,x_n):\sigma \in
\Lambda(x_1,\dots,x_{n-1}) \text{ visible from }x_n\}
$$
where $\sigma$ is \emph{visible} from $x_n$ if
$x_n\notin\cone(x_1,\dots,x_{n-1})$ and the line segment
$[x_n,y]$ for every point $y$ of $\sigma$ intersects
$\cone(x_1,\dots,x_{n-1})$ only in $y$. Note that a polyhedral
complex is closed under the passage to faces, and the
definition above takes care of it.
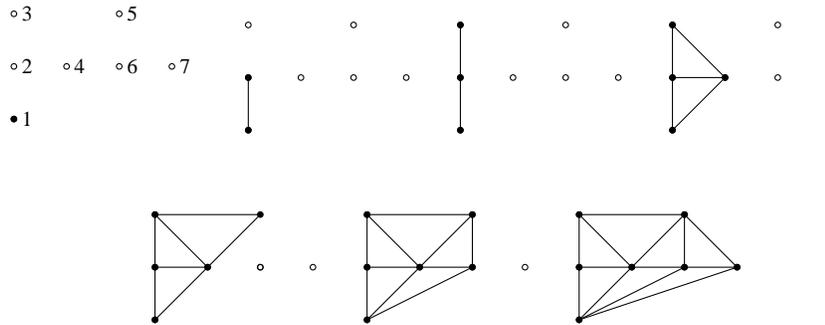
\begin{figure}[hbt]
\begin{center}
\begin{footnotesize}
\begin{tikzpicture}
[scale=0.7,auto=left]
\foreach \x/\y/\z in {0/0/1, 0/1/2, 0/2/3, 1/1/4, 2/1/6, 2/2/5, 3/1/7}
{\draw (\x,\y)  circle (1.5pt) node at (\x.25,\y){\z};}
\foreach \x/\y in {0/0}
{ \filldraw[fill=black] (\x,\y)  circle (1.5pt);}
\end{tikzpicture}
\qquad
\begin{tikzpicture}
[scale=0.7,auto=left]
\foreach \x/\y in {0/0, 0/1, 0/2, 1/1, 2/1, 2/2, 3/1}
{\draw (\x,\y)  circle (1.5pt);}
\foreach \x/\y in {0/0, 0/1}
{ \filldraw[fill=black] (\x,\y)  circle (1.5pt);}
\draw (0,0) -- (0,1);
\end{tikzpicture}
\qquad
\begin{tikzpicture}
[scale=0.7,auto=left]
\foreach \x/\y in {0/0, 0/1, 0/2, 1/1, 2/1, 2/2, 3/1}
{\draw (\x,\y)  circle (1.5pt);}
\foreach \x/\y in {0/0, 0/1, 0/2}
{ \filldraw[fill=black] (\x,\y)  circle (1.5pt);}
\draw (0,0) -- (0,1);
\draw (0,1) -- (0,2);
\end{tikzpicture}
\qquad
\begin{tikzpicture}
[scale=0.7,auto=left]
\foreach \x/\y in {0/0, 0/1, 0/2, 1/1, 2/1, 2/2, 3/1}
{\draw (\x,\y)  circle (1.5pt);}
\foreach \x/\y in {0/0, 0/1, 0/2, 1/1}
{ \filldraw[fill=black] (\x,\y)  circle (1.5pt);}
\draw (0,0) -- (0,1);
\draw (0,1) -- (0,2);
\draw (1,1) -- (0,1);
\draw (1,1) -- (0,0);
\draw (1,1) -- (0,2);
\end{tikzpicture}
\qquad
\vspace*{1cm}

\begin{tikzpicture}
[scale=0.7,auto=left]
\foreach \x/\y in {0/0, 0/1, 0/2, 1/1, 2/1, 2/1, 3/1}
{\draw (\x,\y)  circle (1.5pt);}
\foreach \x/\y in {0/0, 0/1, 0/2, 1/1, 2/2}
{ \filldraw[fill=black] (\x,\y)  circle (1.5pt);}
\draw (0,0) -- (0,1);
\draw (0,1) -- (0,2);
\draw (1,1) -- (0,1);
\draw (1,1) -- (0,0);
\draw (1,1) -- (0,2);
\draw (2,2) -- (0,2);
\draw (2,2) -- (1,1);
\end{tikzpicture}
\qquad
\begin{tikzpicture}
[scale=0.7,auto=left]
\foreach \x/\y in {0/0, 0/1, 0/2, 1/1, 2/1, 2/2, 3/1}
{\draw (\x,\y)  circle (1.5pt);}
\foreach \x/\y in {0/0, 0/1, 0/2, 1/1, 2/1, 2/2}
{ \filldraw[fill=black] (\x,\y)  circle (1.5pt);}
\draw (0,0) -- (0,1);
\draw (0,1) -- (0,2);
\draw (1,1) -- (0,1);
\draw (1,1) -- (0,0);
\draw (1,1) -- (0,2);
\draw (2,1) -- (1,1);
\draw (2,1) -- (0,0);
\draw (1,1) -- (2,2);
\draw (2,2) -- (2,1);
\draw (2,2) -- (0,2);
\end{tikzpicture}
\qquad
\begin{tikzpicture}
[scale=0.7,auto=left]
\foreach \x/\y in {0/0, 0/1, 0/2, 1/1, 2/1, 2/2, 3/1}
{\draw (\x,\y)  circle (1.5pt);}
\foreach \x/\y in {0/0, 0/1, 0/2, 1/1, 2/1, 2/2, 3/1}
{ \filldraw[fill=black] (\x,\y)  circle (1.5pt);}
\draw (0,0) -- (0,1);
\draw (0,1) -- (0,2);
\draw (1,1) -- (0,1);
\draw (1,1) -- (0,0);
\draw (1,1) -- (0,2);
\draw (2,1) -- (1,1);
\draw (2,1) -- (0,0);
\draw (2,2) -- (1,1);
\draw (2,2) -- (2,1);
\draw (2,2) -- (0,2);
\draw (3,1) -- (2,2);
\draw (3,1) -- (0,0);
\draw (3,1) -- (2,1);
\end{tikzpicture}
\end{footnotesize}
\end{center}
\caption{Genesis of a lexicographic triangulation}\label{figlex}
\end{figure}

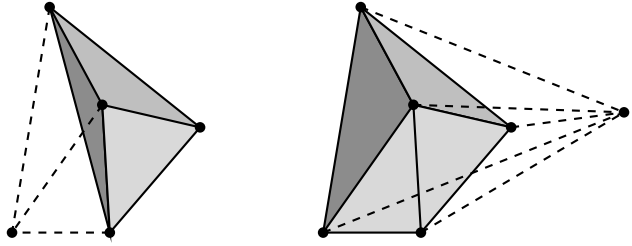
\begin{figure}[hbt]
\begin{center}
\begin{tikzpicture}
\coordinate[] (A) at (0,0);
\coordinate[] (B) at (0.5,3);
\coordinate[] (C) at (1.2,1.7);
\coordinate[] (D) at (2.5,1.4);
\coordinate[] (E) at (4,1.6);
\coordinate[] (F) at (1.3,0);

\draw[thick, draw=black, fill=medium] (B) -- (F) -- (C);
\draw[thick, draw=black, fill=veryverylight] (C) -- (F) -- (D);
\draw[thick, draw=black, fill=verylight] (C) -- (B) -- (D) -- cycle;
\draw[thick, dashed] (B) -- (A);
\draw[thick, dashed] (C) -- (A);
\draw[thick, dashed] (F) -- (A);
\fill (A) circle (2pt);
\fill (B) circle (2pt);
\fill (C) circle (2pt);
\fill (D) circle (2pt);
\fill (F) circle (2pt);
\end{tikzpicture}
\qquad\qquad
\begin{tikzpicture}
\coordinate[] (A) at (0,0);
\coordinate[] (B) at (0.5,3);
\coordinate[] (D) at (2.5,1.4);
\coordinate[] (C) at (1.2,1.7);
\coordinate[] (E) at (4,1.6);
\coordinate[] (F) at (1.3,0);

\draw[thick, fill=medium] (A) -- (B) -- (C);
\draw[thick, draw=black, fill=verylight] (C) -- (B) -- (D) -- cycle;
\draw[thick, fill=veryverylight] (F) -- (A) -- (C) -- (D) -- cycle;

\draw[thick] (C)--(F);

\draw[thick, dashed] (B)--(E);
\draw[thick, dashed] (C)--(E);
\draw[thick, dashed] (D)--(E);
\draw[thick, dashed] (A)--(E);
\draw[thick, dashed] (F)--(E);
\fill (A) circle (2pt);
\fill (B) circle (2pt);
\fill (C) circle (2pt);
\fill (D) circle (2pt);
\fill (E) circle (2pt);
\fill (F) circle (2pt);
\end{tikzpicture}
\end{center}
\caption{A lexicographic triangulation in cone dimension $4$}\label{FigLex2}
\end{figure}

In the algorithms below, a
polyhedral subdivision can always be represented by its maximal
faces which for convex full dimensional polyhedra are the full
dimensional members in the subdivision. For simplicial
subdivisions of cones one uses of course that the face
structure is completely determined by set theory: every subset
$E$ of the set of generators spans a conical face of dimension
$|E|$.

We state some useful properties of lexicographic
triangulations:

\begin{proposition}\label{lex}
With the notation introduced, let $C_i=\cone(x_1,\dots,x_i)$
and $\Lambda_i=\Lambda(x_1,\allowbreak \dots,\allowbreak x_i)$ for
$i=1,\dots,n$.
\begin{enumerate}
\item $\Lambda_n$ is the unique triangulation of $C$  with
    rays through a subset of $\{x_1,\dots,x_n\}$ that
    restricts to a triangulation of $C_i$ for $i=1,\dots,n$
    and $\Lambda|C_i$ has rays through a subset of
    $\{x_1,\dots,x_i\}$.
\item For every face $F$ of $C$ the restriction $\Lambda|F$
    is the lexicographic triangulation
    $\Lambda(x_{i_1},\allowbreak\dots,\allowbreak x_{i_m})$
    where
    $\{x_{i_1},\dots,x_{i_m}\}=F\cap\{x_1,\dots,x_n\}$ and
    $i_1<\dots<i_m$.
\item If $\dim C_i>\dim C_{i-1}$, then
    $\Lambda=\Lambda(x_1,\dots,x_{i-2},x_i,x_{i-1},x_{i+1},\dots,x_n)$.
\item
    $\Lambda=\Lambda(x_{i_1},\dots,x_{i_d},x_{j_1},\dots,x_{j_{n-d}})$
    where $(i_1,\dots,i_d)$ is the lexicographic smallest
    index vector of a rank $d$ subset of
    $\{x_1,\dots,x_n\}$ and $j_1<\dots<j_{n-d}$ lists the
    complementary indices.
\end{enumerate}
\end{proposition}

\begin{proof}
(1) By construction it is clear that $\Lambda_n$ satisfies the
properties of which we claim that they determine $\Lambda$
uniquely. On the other hand, the extension of $\Lambda_{i-1}$
to a triangulation of $C_i$ is uniquely determined if one does
not introduce further rays: the triangulation of the part $V$
of the boundary of $C_{i-1}$ that is visible from $x_i$ has to
coincide with the restriction of $\Lambda_{i-1}$ to $V$.

(2) One easily checks that $\Lambda|F$ satisfies the conditions
in (1) that characterize $\Lambda(x_{i_1},\dots,\allowbreak
x_{i_m})$.

(3) It is enough to check the claim for $i=n$. Then the only
critical point for the conditions in (1) is whether
$\Lambda(x_1,\dots,x_{n-2},x_n,\allowbreak x_{n-1})$ restricts
to $C_{n-1}$. But this is the case since $C_{n-1}$ is a facet
of $C$ if $\dim C>\dim C_{n-1}$.

(4) follows by repeated application of (3).
\end{proof}

For the configuration of Figure \ref{figlex}, claim 4 of Proposition \ref{lex} says that we could have started with the triangle spanned by the points 1,2,4 and then added the other points in the given order.

In the following we will assume that $C$ is full dimensional:
$\dim C=d=\dim \RR^d$. Part (4) helps us to keep the data
structure of lexicographic triangulations simple: right from
the start we need only to work with the list of dimension $d$
simplicial cones of $\Lambda$ by searching
$x_{i_1},\dots,x_{i_d}$ first, choosing
$\cone(x_{i_1},\dots,x_{i_d})$ as the first $d$-dimensional
simplicial cone and subsequently extending the list as
prescribed by the definition of the lexicographic
triangulation. In other words, we can assume that
$x_1,\dots,x_d$ are linearly independent, and henceforth we
will do so.

In order to extend the triangulation we must of course know
which facets of $C_{i-1}$ are visible from $x_i$. Recall that a
cone $C$ of dimension $d$ in $\RR^d$ has a unique irredundant
representation as an intersection of linear halfspaces:
$$
C=\bigcap_{H\in\cH(C)} H^+,
$$
where $\cH(C)$ is a finite set of oriented hyperplanes and the
orientation of the closed half spaces $H^-$ and $H^+$ is chosen
in such a way that $C\subset H^+$ for $H\in\cH(C)$. For
$H\in\cH(C_{i-1})$ the facet $H\cap C_{i-1}$ is visible from
$x_i$ if and only if $x_i$ lies in the open halfspace
$H^<=H^-\setminus H$. When we refer to support hyperplane in
the following we always mean those that appear in the
irredundant decomposition of $C$ since only they are important
in the algorithmic context.

Hyperplanes are represented by linear forms
$\lambda\in(\RR^d)^*$, and we always work with the basis
$e_1^*,\dots,e_d^*$ that is dual to the basis $e_1,\dots,e_d$
of unit vectors. For rational hyperplanes the linear form
$\lambda$ can always be chosen in such a way that it has
integral coprime coefficients and satisfies $\lambda(x)\ge 0$
for $x\in C$. This choice determines $\lambda$ uniquely. (If
one identifies $e_1^*,\dots,e_d^*$ with $e_1,\dots,e_d$ via the
standard scalar product, then $\lambda$ is nothing but the
primitive integral inner (with respect to $C$) normal vector of
$H$.) For later use we define the \emph{(lattice) height} of
$x\in\RR^d$ over $H$ by
$$
\hht_H(x)=|\lambda(x)|.
$$
If $F=C\cap H$ is the facet of $C$ cut out by $H$, we set
$\hht_F(x)=\hht_H(x)$.

We can now describe the computation of the triangulation
$\Lambda(x_1,\dots,x_n)$ and the support hyperplanes in a more formal way by Algorithm
\ref{Lexincr}. For simplicity we will identify a simplicial
cone $\sigma$ with its generating set
$\subset\{x_1,\dots,x_n\}$. It should be clear from the context
what is meant. For a set $\cH$ of hyperplanes we set
$$
\cH^*(x)=\{H\in \cH, x\in H^*\}\qquad\text{where }*\in\{<,>,+,-\}.
$$
Further we introduce the notation
$$
\cH^*(C,x)=\{H\in \cH(C), x\in H^*\}\qquad\text{where }*\in\{<,>,+,-\}.
$$

The representation of hyperplanes by linear forms makes it easy to detect the visible facets: a facet is visible from $y$ if $\lambda(y)<0$ for the linear from $\lambda$ defining the hyperplane through the facet.
As pointed out above, in Algorithm \ref{Lexincr} and at several places below we may assume that the first $d$ elements of $x_1,\dots,x_n$ are linearly independent. This can always be achieved by rearranging the order of the elements, or by a refined bookkeeping (as done by Normaliz).

For its main data, Normaliz uses two types of data structures:
\begin{enumerate}
\item Lists and matrices of integer vectors. The vectors represent generators of cones, Hilbert basis elements etc.\ in $\RR^d$, or linear forms in $(\RR^d)^*$.
\item Lists  of subsets of the set $\{x_1,\dots,x_n\}$. Each subset stands for the subcone generated by its elements.
\end{enumerate}
Sometimes more complicated data structures are needed. For example, it is useful in Algorithm \ref{Lexincr} to store the incidence relation of generators and facets.

\begin{algorithm}[hbt]
  \begin{algorithmic}[1]
	\Require{A generating set $x_1,\dots,x_n$ of a rational cone $C$ of dimension $d$}
	\Ensure{The support hyperplanes $\cH$ of $C$ and the triangulation $\Lambda(x_1,\dots,x_n)$}
  \Function{LexTriangulation}{$x_1,\dots,x_n$}
    \State{$\Delta \gets \{\cone(x_1,\dots,x_d)\}$}
    \State{$\cH\gets\cH(\cone(x_1,\dots,x_d))$}
    \For{$i \gets d+1$ \textbf{to} $n$}
      \State{$\Delta\gets$\Call{ExtendTri}{$\cH,\Delta,x_i$}}
      \State{$\cH\gets$\Call{FindNewHyp}{$\cH,x_1,\dots,x_i$}}
    \EndFor
    \State{\Return{$(\cH,\Delta)$}}
  \EndFunction
  \end{algorithmic}
  \begin{algorithmic}[1]
  	\Require{A set of hyperplanes $\cH$, a triangulation $\Delta$ and a point $y$}
  	\Ensure{The union of $\Delta$ with the set of simplicial cones spanned by $y$ and the  facets $\delta$ of the $\sigma\in\Delta$ such that $\delta\subset H$ for some $H\in\cH$ with $y\in H^<$ }
  \Function{ExtendTri}{$\cH,\Delta,y$}
  \ParFor{$H\in \cH^<(y)$}
    \For{$\sigma\in \Delta$}
      \If{$|\sigma\cap H|=d-1$}
        \State{$\Delta \gets \Delta\cup \{\cone(y,\sigma\cap H)\}$}
      \EndIf
    \EndFor
  \EndParFor
  \State{\Return $\Delta$}
  \EndFunction
  \end{algorithmic}
\caption{Incremental building of cone, support hyperplanes  and lexicographic triangulation}
\label{Lexincr}
\end{algorithm}

In the following discussion we set $C_j=\cone(x_1,\dots,x_j)$ as above. The support hyperplanes of the first simplicial cone $C_d$ in line 3 are computed by essentially inverting the matrix of the
generators $x_1,\dots,x_d$ (see equation \eqref{invert} in Section \ref{Eval}). The function \textsc{FindNewHyp} computes $\cH(C_i)$ from
$\cH(C_{i-1})$ by Fourier-Motzkin elimination. (It does nothing
if $x_i\in C_{i-1}$.) Its Normaliz implementation has been
described in great detail in \cite{BI}; therefore we skip it
here, but will come back to it below when we outline its combination with pyramid decomposition. The function \textsc{ExtendTri} does exactly what its
name says: it extends the triangulation
$\Lambda(x_1,\dots,x_{i-1})$  of $C_{i-1}$ to the triangulation
$\Lambda(x_1,\dots,x_{i})$ of $C_i$ (again doing nothing if
$x_i\in C_{i-1}$).

One is tempted to improve \textsc{ExtendTri} by better
bookkeeping and using extra information on triangulations of
cones. We discuss our more or less fruitless attempts in the
following remark.

\begin{remark} (a) If one knows the restriction of
$\Lambda(x_1,\dots,x_{i-1})$ to the facets of
$C_{i-1}$, then $\Lambda(x_1,\dots,x_{i})$ can be computed very
fast. However, unless $i=n$, the facet triangulation must now
be extended to the facets of $C_i$, and this step eats up the
previous gain, as experiments have shown, at least for the
relatively small triangulations to which \textsc{ExtendTri} is
really applied after the pyramid decomposition described below.

(b) The test of the condition $|\sigma\cap H|=d-1$ is positive
if and only if $d-1$ of the generators of $\sigma$ lie in $H$.
Its verification can be accelerated if one knows which facets
of the $d$-dimensional cones in $\Lambda(x_1,\dots,x_{i-1})$
are already shared by another simplicial cone in
$\Lambda(x_1,\dots,x_{i-1})$, and are therefore not available
for the formation of a new simplicial cone. But the extra
bookkeeping requires more time than is gained by its use.

(c) One refinement is used in our implementation, though its
influence is almost unmeasurable. Each simplicial cone in
$\Lambda(x_1,\dots,x_{i-1})$ has been added with a certain
generator $x_j$, $j<i$. (The first cone is considered to be
added with each of its generators.) It is not hard to see that
only those simplicial cones that have been added with a
generator $x_j\in H$ can satisfy the condition $|\sigma\cap
H|=d-1$, and this information is used to reduce the number of
pairs $(H,\sigma)$ to be tested.

(d) If $|H\cap \{x_1,\dots,x_{i-1}\}|=d-1$, then
$H\in\cH^<(C_{i-1},x_i)$ produces exactly one new simplicial
cone of dimension $d$, namely
$\cone(x_i,H\cap\{x_1,\dots,x_{i-1}\})$, and therefore the loop
over $\sigma$ can be suppressed.
\end{remark}

The product $|\cH^<(C_{i-1},x_i)|\cdot|\Lambda(x_1,\dots,x_{i-1})|$ determines the
complexity of \textsc{ExtendTri}. Even though the loop over $H$
is parallelized (as indicated by \textbf{parallel for}), the
time spent in \textsc{ExtendTri} can be very long. (The
``exterior'' loops in \textsc{FindNewHyp} are parallelized as
well.) The second limiting factor for \textsc{ExtendTri} is
memory: it is already difficult to store triangulations of size
$10^8$ and impossible for size $\ge10^9$. Therefore the direct
approach to lexicographic triangulations does not work for
truly large cones.

\begin{remark}
The computation time for the Fourier-Motzkin elimination and the lexicographic triangulation often depends significantly on the order of the generators. If only the support hyperplanes must be computed, Normaliz orders the input vectors lexicographically. If also the triangulation must be computed, the input vectors are first sorted by their $L_1$-norm, or by degree if a grading is defined (see Section \ref{Eval}), and second lexicographically. The sorting by $L_1$-norm or degree helps to keep the determinants of the simplicial cones small (see Section \ref{Eval}).  On the whole, we have reached good results with this order.
\end{remark}

\begin{remark}
Whenever possible, each parallel thread started in a Normaliz computation collects its computation results and returns them to the calling routine after its completion. In this way, the amount of synchronization between the threads is reduced to a minimum. For example, in \textsc{ExtendTri}, the new simplicial cones $\cone(y,\sigma\cap H)$ can be collected independently of each other: they are not directly added to the global list $\Delta$ in line 9, but are first stored in a list owned by the thread, and then spliced into $\Delta$ at the end of \textsc{ExtendTri}. 	
\end{remark}

\subsection{Pyramid decomposition}

Now we present a radically different way to lexicographic
triangulations via iterated \emph{pyramid decompositions}. The
cones that appear in this type of decomposition are called
\emph{pyramids} since their cross-section polytopes are
pyramids in the usual sense, namely of type $\conv(F,x)$ where
$F$ is a facet and $x$ is a vertex not contained in $F$.

\begin{definition}
The \emph{pyramid decomposition} $\Pi(x_1,\dots,x_n)$ of
$C=\cone(x_1,\dots,x_n)$ is recursively defined as follows: it
is the trivial decomposition for $n=0$, and
\begin{multline*}
\Pi(x_1,\dots,x_n)=\Pi(x_1,\dots,x_{n-1})\cup \{
\cone(F,x_n):\\ F \text{ a face of $\cone(x_1,\dots,x_{n-1})$
visible from }x_n\}.
\end{multline*}
\end{definition}

As already pointed out in the introduction, the pyramid decomposition is
not a polyhedral
subdivision in the strong sense: the intersection of two faces
$F$ and $F'$ need not be a common face of $F$ and $F'$ (but is
always a face of $F$ or $F'$). See Figures \ref{FigPyrDec} and \ref{FigPyrDec2} for
examples. Roughly speaking, one can say that in the pyramid decomposition forgets the potentially existing subdivision (or even triangulation) of the facets of $C(x_1,\dots,x_{n-1}$ that are visible from $x_n$. In order to subdivide (or even triangulate)  the new pyramids it is enough to do the computations within each of them. This ``localization'' reduces the complexity tremendously.

In order to iterate the pyramid decomposition we set
$\Pi^0(x_1,\dots,x_n)=\Pi(x_1,\dots,x_n)$, and
$$
\Pi^k(x_1,\dots,x_n) =\bigcup_{P\in\Pi^{k-1}(x_1,\dots,x_n)}\{ \Pi(x_i: x_i\in
P)\}\qquad\text{for }k>0.
$$
We now assume that the first $d$ vectors in the generating set of the top cone and each of its pyramids are linearly independent. Because of Proposition \ref{lex}, claim 4, this assumption does not endanger the compatibility with lexicographic triangulation. Under this
assumption the recursion defining $\Pi^k$ cannot descend indefinitely, since the
number of generators goes down with each recursion  level. We
denote the \emph{total pyramid decomposition} by
$\Pi^\infty(x_1,\dots,\allowbreak x_n)$.
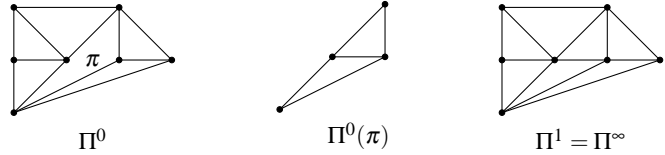
\begin{figure}[hbt]
\begin{center}
\begin{small}
\begin{tikzpicture}
[scale=0.7,auto=left]
\foreach \x/\y in {0/0, 0/1, 0/2, 1/1, 2/1, 2/2, 3/1}
{ \filldraw[fill=black] (\x,\y)  circle (1.5pt);}
\draw (0,0) -- (0,1) node at (1.5,-0.5){$\Pi^0$};
\draw (0,1) -- (0,2) node at (1.5,1){$\pi$};
\draw (1,1) -- (0,0);
\draw (1,1) -- (0,2);
\draw (2,1) -- (0,0);
\draw (2,2) -- (1,1);
\draw (2,2) -- (2,1);
\draw (2,2) -- (0,2);
\draw (3,1) -- (2,2);
\draw (3,1) -- (0,0);
\draw (3,1) -- (2,1);
\draw (0,1) -- (1,1);
\end{tikzpicture}
\qquad\qquad
\begin{tikzpicture}
[scale=0.7,auto=left]
\foreach \x/\y in {0/0, 1/1, 2/1, 2/2}
{ \filldraw[fill=black] (\x,\y)  circle (1.5pt);}
\draw (2,1) -- (0,0) node at (1.5,-0.5){$\Pi^0(\pi)$};
\draw (2,2) -- (1,1);
\draw (2,2) -- (2,1);
\draw (1,1) -- (0,0);
\draw (1,1) -- (2,1);
\end{tikzpicture}
\qquad\qquad
\begin{tikzpicture}
[scale=0.7,auto=left]
\foreach \x/\y in {0/0, 0/1, 0/2, 1/1, 2/1, 2/2, 3/1}
{ \filldraw[fill=black] (\x,\y)  circle (1.5pt);}
\draw (0,0) -- (0,1) node at (1.5,-0.5){$\Pi^1=\Pi^\infty$};
\draw (0,1) -- (0,2);
\draw (1,1) -- (0,1);
\draw (1,1) -- (0,0);
\draw (1,1) -- (0,2);
\draw (2,1) -- (1,1);
\draw (2,1) -- (0,0);
\draw (2,2) -- (1,1);
\draw (2,2) -- (2,1);
\draw (2,2) -- (0,2);
\draw (3,1) -- (2,2);
\draw (3,1) -- (0,0);
\draw (3,1) -- (2,1);
\end{tikzpicture}
\end{small}
\end{center}
\caption{Pyramid decomposition of the point configuration of Figure \ref{figlex}}\label{FigPyrDec}
\end{figure}

\begin{figure}[hbt]
\begin{center}
\begin{tikzpicture}
\coordinate[] (A) at (0,0);
\coordinate[] (B) at (0.5,3);
\coordinate[] (C) at (1.2,1.7);
\coordinate[] (D) at (2.5,1.4);
\coordinate[] (E) at (4,1.6);
\coordinate[] (F) at (1.3,0);

\draw[thick, draw=black, fill=medium] (B) -- (F) -- (C);
\draw[thick, draw=black, fill=veryverylight] (C) -- (F) -- (D);
\draw[thick, draw=black, fill=verylight] (C) -- (B) -- (D) -- cycle;
\draw[thick, dashed] (B) -- (A);
\draw[thick, dashed] (C) -- (A);
\draw[thick, dashed] (F) -- (A);
\fill (A) circle (2pt);
\fill (B) circle (2pt);
\fill (C) circle (2pt);
\fill (D) circle (2pt);
\fill (F) circle (2pt);
\end{tikzpicture}
\qquad\qquad
\begin{tikzpicture}
\coordinate[] (A) at (0,0);
\coordinate[] (B) at (0.5,3);
\coordinate[] (D) at (2.5,1.4);
\coordinate[] (C) at (1.2,1.7);
\coordinate[] (E) at (4,1.6);
\coordinate[] (F) at (1.3,0);

\draw[thick, fill=medium] (A) -- (B) -- (C);
\draw[thick, draw=black, fill=verylight] (C) -- (B) -- (D) -- cycle;
\draw[thick, fill=veryverylight] (F) -- (A) -- (C) -- (D) -- cycle;

\draw[thick, dashed] (B)--(E);
\draw[thick, dashed] (C)--(E);
\draw[thick, dashed] (D)--(E);
\draw[thick, dashed] (A)--(E);
\draw[thick, dashed] (F)--(E);
\fill (A) circle (2pt);
\fill (B) circle (2pt);
\fill (C) circle (2pt);
\fill (D) circle (2pt);
\fill (E) circle (2pt);
\fill (F) circle (2pt);
\end{tikzpicture}
\end{center}
\caption{Pyramid decomposition of Figure \ref{FigLex2} }\label{FigPyrDec2}
\end{figure}
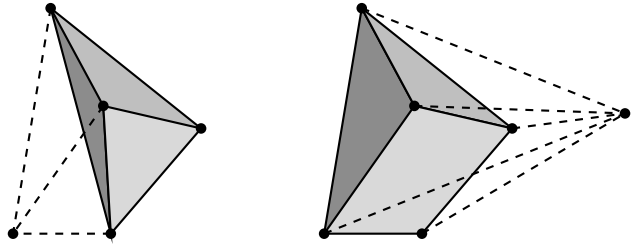

\begin{proposition}
One has
$\Pi^\infty(x_1,\dots,x_n)=\Pi^{n-d}(x_1,\dots,x_n)=\Lambda(x_1,\dots,x_n)$.
\end{proposition}

\begin{proof}
In the case $n=d$, the pyramid decomposition is obviously the
face lattice of $C$, and therefore coincides with the
lexicographic triangulation. For $n>d$ the first full
dimensional pyramid reached is the simplicial cone
$\cone(x_1,\dots,x_d)$. All the other pyramids have at most
$n-1$ generators, and so we can use induction: For each $P\in
\Pi(x_1,\dots,x_n)$ the total pyramid decomposition of $P$ is
the lexicographic triangulation $\Lambda(x_i: x_i\in P)$.
According to Proposition \ref{lex}(2) these triangulations
match along the common boundaries of the pyramids, and
therefore constitute a triangulation of $C$. It evidently
satisfies the conditions in Proposition \ref{lex}(1).
\end{proof}

This  leads to a recursive computation of
$\Lambda(x_1,\dots,x_n)$ by the functions in Algorithm~\ref{TotalPyrDec}.

\begin{algorithm}[hbt]
  \begin{algorithmic}[1]
  \Require{A generating set $x_1,\dots,x_n$ of a rational cone $C$ of dimension $d$}
  \Ensure{The support hyperplanes $\cH$ and of $C$ and the triangulation $\Lambda(x_1,\dots,x_n)$}
    \Function{TotalPyrDec}{$x_1,\dots,x_n$}
      \State{$\Delta\gets\{\cone(x_1,\dots,x_d)\}$}
      \State{$\cH\gets\cH(\cone(x_1,\dots,x_d))$}
      \For{$i \gets d+1$ \textbf{to} $n$}
        \State{($\cG,\Sigma)\gets$\Call{ProcessPyrsRec}{$\cH,x_1,\dots,x_i$}}
        \State{$\cH \gets (\cH\cup \cG)\setminus \cH^<(x_i)$}
        \State{$\Delta \gets \Delta\cup\Sigma$}
      \EndFor
	  \State{\Return{$(\cH,\Delta$)}}
    \EndFunction
  \end{algorithmic}

  \begin{algorithmic}[1]
     \Require{A generating set $x_1,\dots,x_i$ of a rational cone $C$ and the support hyperplanes $\cH=\cH(\cone(x_1,\dots,x_{i-1})$}
     \Ensure{The support hyperplanes $\cH(x_1,\dots,x_n)\setminus\cH(x_1,\dots,x_{n-1})$ and  the triangulation $\Lambda(x_1,\dots,x_n)\setminus \Lambda(x_1,\dots,x_{n-1})$}
  	\Function{ProcessPyrsRec}{$\cH,x_1,\dots,x_n$}
  	\State{$\Delta\gets\emptyset$}
  	\State{$\cG\gets\emptyset$}
  	\ParFor{$H\in \cH^<(x_n)$}
  	\State{$key \gets \{x_n\}\cup(\{x_1,\dots,x_{n-1}\}\cap H)$}
  	\State{$(\cK,\Sigma)\gets$\Call{TotalPyrDec}{\textit{key}}}
  	\State{$\cG\gets \cG\cup \{G\in\cK:G\in\cH(\cone(x_1,\dots,x_n))\}$}
  	\State{$\Delta \gets \Delta\cup\Sigma$}
  	\EndParFor
  	\State{\Return{$(\cG,\Delta$)}}
  	\EndFunction
  \end{algorithmic}
\caption{Incremental building of cone, support hyperplanes and lexicographic triangulation by total pyramid decomposition}\label{TotalPyrDec}
\end{algorithm}

When called with the arguments $x_1,\dots,x_n$, the function \textsc{TotalPyrDec} builds $\Pi^\infty(x_1,\dots,x_n)$ (represented by its full dimensional members). As in Algorithm \ref{Lexincr}, the support hyperplanes of the simplicial cone $C_d$ in line 3 are computed by the inversion of the generator matrix. All further support hyperplanes are given back to $C_n$ by its ``daughters'' in line 6 where we also discard the support hyperplanes of $C_{n-1}$ that have $x_i$ in their negative half space.

The function \textsc{ProcessPyrsRec} manages the recursion that defines $\Pi^\infty(x_1,\dots,\allowbreak x_n)$. In its line 7 we must decide which support hyperplanes $G$ of the daughter pyramid $\cone(\textit{key})$ are ``new'' support hyperplanes of the mother $C_n=\cone(x_1,\dots,x_n)$. We use the following criteria:
\begin{enumerate}
	\item[(i)] $G \in\cH(C_n)\iff x_j\in G^+$ for $j=1,\dots,n-1$;
	\item[(ii)] $G\notin
	\cH(C_{n-1})\iff x_j\in G^>$ for all $j=1,\dots,i-1$ such that 	$x_j\notin \textit{key}$.
\end{enumerate}

One should note that pyramids effectively reduce the dimension:
the complexity of $\cone(F,x_n)$ is completely determined by
the facet $F$, which has dimension $d-1$.

While pyramid decomposition has primarily been developed for the computation of triangulations, it is also very useful in the computation of support hyperplanes. For Fourier-Motzkin elimination the critical complexity parameter
is $|\cH^<(C_{i-1},x_i)|\cdot|\cH^>(C_{i-1},x_i)|$, and as in
its use for triangulation, pyramid decomposition lets us
replace a potentially very large product of the sizes of two ``global''
lists by a sum of small ``local'' products--the price to be
paid is the computational waste invested for the support
hyperplanes of the pyramids that are useless later on.

While being very memory efficient, total pyramid
decomposition in the na\"{i}ve implementation of Algorithm \ref{TotalPyrDec} is sometimes slower and sometimes faster than using Fourier-Motzkin elimination and building the
lexicographic triangulation directly. The best solution is a hybrid algorithm that combines pyramid decomposition and lexicographic triangulation. It will be descried in the next section where we will also compare computation times and memory usage of pure lexicographic triangulation, pure pyramid decomposition and the hybrid algorithm. We compare computation times in Section \ref{comp_times}.

\section{The current implementation}\label{Curr}

\subsection{The hybrid algorithm}

Roughly speaking, the hybrid algorithm switches from Fourier-Motzkin elimination and lexicographic triangulation to pyramid decomposition for hyperplanes and triangulation when certain complexity parameters are exceeded. This strategy is realized by the function \textsc{BuildCone} of Algorithm \ref{NewBuildCone}.

\begin{algorithm}[hbt]
\begin{algorithmic}[1]
\Require{A generating set $x_1,\dots,x_n$ of a rational cone $C$ of dimension $d$}. The top cone has an initially empty list  $\Pi$ of pyramids.
\Ensure{The support hyperplanes $\cH$ and of $C$ and the triangulation $\Lambda(x_1,\dots,x_n)$}
\Function{BuildCone}{$x_1,\dots,x_n$}
      \State{$\Delta\gets\{\cone(x_1,\dots,x_d)\}$}
      \State{$\cH\gets\cH(\cone(x_1,\dots,x_d))$}
  \For{$i \gets d+1$ \textbf{to} $n$}
    \If{\textit{MakePyramidsForHyps}}
        \State{($\cG,\Sigma)\gets$\Call{ProcessPyrsRec}{$\cH,x_1,\dots,x_i$}}
        \State{$\cH \gets (\cH\cup \cG)\setminus \cH^<(x_i)$}
        \State{$\Delta \gets \Delta\cup\Sigma$}
    \Else
      \If{\textit{MakePyramidsForTri}}
          \For{$H\in \cH^<(\cH, x_i)$}
          \State{$key \gets \{x_i\}\cup(\{x_1,\dots,x_{i-1}\}\cap H)$}
          \State{$\Pi\gets\Pi\cup\{\textit{key} \}$}
          \EndFor
      \Else
        \State{$\Delta\gets$\Call{ExtendTri}{$\cH,\Delta,x_i$}}
      \EndIf
      \State{$\cH\gets$\Call{FindNewHyp}{$\cH,x_1,\dots,x_i$}}
    \EndIf
  \EndFor
\If{\textit{TopCone}}
  \ParFor{$P\in \Pi$}
  \State{\Call{buildCone}{$P$}}
  \State{$\Pi\gets\Pi\setminus\{P\}$}
  \EndParFor
\EndIf
\State{\Return $(\cH,\Delta)$}

\EndFunction
 \end{algorithmic}
\caption{Incremental building of cone, support hyperplanes and lexicographic triangulation by a hybrid algorithm}\label{NewBuildCone}
\end{algorithm}

The boolean \textit{MakePyramidsForHyps} (line 5) is determined by a single condition:
\begin{itemize}
\item[] it is set to \textit{true} if the complexity parameter
    $|\cH^<(C_{i-1},x_i)|\cdot|\cH^>(C_{i-1},x_i)|$ exceeds
    a threshold, and to \textit{false} otherwise.
\end{itemize}
As the name \textit{MakePyramidsForHyps} indicates,thew computation of support hyperplanes is transferred to the pyramids over the hyperplanes $\cH^<(x_i)$ if the complexity parameter is exceeded. Pyramids created for the computation of support hyperplanes must be treated very carefully since the mother cone must wait for the computation of their support hyperplanes. We come back to this point below.

The \textit{MakePyramidsForTri} (line 10) combines three conditions:
\begin{enumerate}
\item while set to \textit{false} initially, it remains
    \textit{true} once  once the switch to pyramids has been done in line 5 or line 10;
\item it is set \textit{true} if the complexity parameter
    $|\cH^<(C_{i-1},x_i)|\cdot|\Delta|$ exceeds a
    threshold;
\item it is set \textit{true} if the memory protection
    threshold is exceeded.
\end{enumerate}
The last point needs to be explained. \textsc{BuildCone} is not
only called for the processing of the top cone $C$, but also
for the parallelized processing of pyramids. Since
each of the ``parallel'' pyramids produces simplicial cones,
the buffer in which the simplicial cones are collected for
evaluation, may be severely overrun without condition (3),
especially if $|\cH^<(x_i)|$ is small, and therefore
condition (2) is reached only for large
$|\Lambda(x_1,\dots,x_{i-1})|$.

Pyramids that are created for triangulation can simply be stored since their triangulation is not needed for the continuation of the pyramid decomposition. Line 13 of \textsc{BuildCone} therefore  adds them to the pyramid list $\Pi$ which is part of the data of the top cone. The stored pyramids are evaluated after the top cone has been completely built (lines 17--20). It is a crucial aspect of pyramid decomposition that the loop in lines 18--20 is parallelized: the evaluation of a pyramid is a completely independent computation.

In the triangulation of the stored pyramids, new daughter pyramids may be created and added to the list. However, the number of pyramids is is bounded by $|\Lambda(x_1,\dots,x_n)|$. At its termination,\textsc{BuildCone} returns the support hyperplanes of the top cone and the lexicographic triangulation $\Lambda(x_1,\dots,x_n)$.

Algorithm \ref{NewBuildCone} is only a structural model of the actual implementation. Some of its technical details will be described below.

\subsection{Pyramids for support hyperplanes}

Pyramids that have been created because of the complexity of Fourier-Motzkin elimination are treated by the function \textsc{ProcessPyrsRec}. The \textsc{Rec} in its name indicates that the computation of the mother cone must wait for the completion of the daughter pyramid, at least for its support hyperplanes.
\begin{algorithm}[hbt]
  \begin{algorithmic}[1]
  	\Require{A generating set $x_1,\dots,x_i$ of a rational cone $C$ and the support hyperplanes $\cH=\cH(\cone(x_1,\dots,x_{i-1}))$}
  	\Ensure{The support hyperplanes $\cH(x_1,\dots,x_i)\setminus\cH(x_1,\dots,x_{i-1})$ and  part of the triangulation $\Lambda(x_1,\dots,x_i)\setminus \Lambda(x_1,\dots,x_{i-1})$}
  	\Function{ProcessPyrsRec}{$\cH,x_1,\dots,x_i$}
  	\State{$\Delta\gets\emptyset$}
  	\State{$\cG\gets\emptyset$}
  	\ParFor{$H\in \cH^<(x_i)$}
  	\State{$key \gets \{x_i\}\cup(\{x_1,\dots,x_{i-1}\}\cap H)$}
  \If{\textit{Small}}
  	\State{$(\cK,\Sigma)\gets$\Call{BuildCone}{\textit{key}}}
  	\State{$\cG\gets \cG\cup \{G\in\cK:G\in\cH(\cone(x_1,\dots,x_i))\}$}
  	\State{$\Delta \gets \Delta\cup\Sigma$}
  \Else
      \State{$\cG\gets \cG\ \cup$ \Call{MatchWitPosHyps}{$H,\cH,x_1.\dots,x_i$}  }
      \State{$\Pi\gets\Pi\cup\{\textit{key} \}$}
  \EndIf
  	\EndParFor
  	\State{\Return{$(\cG,\Delta$)}}
  	\EndFunction
  \end{algorithmic}
\caption{Processing of pyramids towards support hyperplanes and triangulation of mother cone}\label{ProcPyrsRec}
\end{algorithm}

The function is similar to the function \textsc{ProcessPyrsRec} in Algorithm \ref{TotalPyrDec}, except that we now distinguish between ``small'' and ``large'' pyramids. Small pyramids are treated recursively as in the total pyramid decomposition, namely by applying \textsc{BuildCone } to them. The treatment of large pyramids differs in two ways:
\begin{enumerate}
\item the triangulation of the pyramid is deferred;
\item the Fourier-Motzkin step \textsc{MatchWitPosHyps} is used to find the support hyperplanes of the mother cone that originate from $H$.
\end{enumerate}

The criterion for \textit{small} is based on a comparison of the expected computation times for (i) building the pyramid over $H$ and (ii) the Fourier-Motzkin step in which $H$ is ``matched'' with the hyperplanes $G\in \cH^>(x_i)$; see \cite{BI}. This refinement was the last step added to the processing of pyramids. It is irrelevant in sequential computations, but large pyramids previously had the tendency to significantly delay the completion of the parallelized loop in line 4.

\subsection{Interruption strategy}

Normaliz keeps all data in in RAM. Therefore it is necessary to control the size of the lists that contain simplicial cones and pyramids. This is achieved by a strategy that interrupts the production of pyramids and simplicial cones at suitable points as soon as the lists sizes have exceeded a preset value. The choice of the interruption points must take into consideration that Normaliz avoids nested
parallelization for efficiency. (This is the default choice of OpenMP.)

As soon as \textsc{BuildCone} switches to pyramids, the
triangulation $\Lambda(x_1,\dots,x_{i-1})$ is no longer needed for further extension. Therefore it is shipped to the
evaluation buffer. The simplicial cones are evaluated and the buffer is emptied whenever it has
exceeded its preset size and program flow allows its
parallelized evaluation.

The strategy for the evaluation of pyramids is similar, but it takes into account the recursive nature of the pyramid decomposition. The pyramid list is actually split into levels, and pyramids of level $i$ produce subpyramids of level $i+1$. If the number of level $i+1$ pyramids becomes too large, the production at level $i$ is interrupted in favor of the processing of the level $i+1$ pyramids.

\subsection{Partial triangulation}\label{partial}

The idea of pyramid
decomposition was born when the authors observed that the
computation of Hilbert bases in principle does not need a full
triangulation of $C$. If a simplicial cone $\sigma$ cannot
contribute new candidates for the Hilbert basis of $C$, it need
not be evaluated, and if a pyramid consists only of such
simplicial cones, it need not be triangulated at all. This is
the case if $\hht_H(x_i)=1$.

The resulting strategy has sometimes striking results and was
already described in~\cite{BHIKS}.

\subsection{Computation times}\label{comp_times}

Section \ref{Comp} contains extensive data on the performance of Normaliz. The computation times listed there include the evaluation of the simplicial cones for Hilbert bases and Hilbert series using the hybrid algorithm.

Here we want to compare lexicographic triangulation/Fourier-Motzkin elimination, pure pyramid decomposition and the hybrid algorithm in the computation of triangulations and support hyperplanes and triangulations, excluding any evaluation. (Normaliz can be restricted to these tasks.) The sources of the test input files pf Table \ref{data_pyrdec} are listed in Section \ref{Comp} where we give computation times for a large number of examples. The times reported in this section were taken on a SUN xFire 4450 with with 4 Intel
Xeon X7460 (a total of 24 cores running at 2.66 GHz) and 128 GB RAM.

\begin{small}
\begin{table}[hbt]
\centering
\begin{tabular}{|l|r|r|r|r|r|}\hline
\rule[-0.1ex]{0ex}{2.5ex}Input& edim & rank & $\#$ext & $\#$supp& $\#$ triangulation\\ \hline
\strut \ttt{CondPar} & 24 & 24 & 234 & 27 & 1,344,671 \\ \hline
\strut \ttt{5x5}     & 25 & 15 & 1,940 & 25 & 14,615,011\\ \hline
\strut \ttt{lo6}     & 16 & 16 & 720 & 910 & 5,796,124,824 \\ \hline
\strut \ttt{cyclo60} & 17 & 17 & 60 & 656,100 & 11,741,300 \\ \hline
\strut \ttt{A443}    & 40 & 30 & 48 & 4,948 & 2,654,272 \\ \hline
\strut \ttt{A543}    & 47 & 36 & 60 & 29,387  & 102,538,890  \\ \hline
\strut \ttt{A553}    & 55 & 43 & 75 & 306,955 & 9,248,466,183 \\ \hline
\end{tabular}
\vspace*{2ex} \caption{Numerical data of test
examples}\label{data_pyrdec}
\end{table}
\end{small}

As Table \ref{data_tri} shows, the hybrid algorithm is far superior to lexicographic triangulation as soon as the triangulations are large enough to have pyramids really built. Moreover, the need of storing the whole triangulation in RAM limits the applicability of lexicographic triangulation to sizes of $\approx 10^8$: \ttt{A543} needs already $21$ GB of RAM, and therefore \ttt{lo6} and \ttt{A553} cannot be computed by it, even if one is willing to wait for a very long time. The RAM needed by the hybrid algorithm is essentially determined by the fact that Normaliz collects $2.5 \cdot 10^6$ simplicial cones for parallelized evaluation, and is typically between $500$ MB and $1$ GB.

When the number of support hyperplanes is very large relative to the triangulation size, as for \ttt{cyclo60}, total pyramid decomposition is much better than lexicographic triangulation and can compete with the hybrid algorithm. This is not surprising since the pyramids built by the hybrid algorithm are close to being simplicial. The efficiency of parallelization depends on the use of \textsc{ProcessPyrsRec}: the dependence of the mother on the daughters limits the gain by parallelization.

\begin{small}
\begin{table}[hbt]
\centering
\begin{tabular}{|l|r|r|r|r|}\hline
\rule[-0.1ex]{0ex}{2.5ex}Input&threads&lex triang&total pyr dec& hybrid\\ \hline
\strut \ttt{CondPar} & 1& 15.8 s & 2:06 m & 3.0 s\\ \hline
& 20& 10.5 s& 1:20 m & 2.8 s \\ \hline
\strut \ttt{A443} & 1&  8:32 m& 4:37 m & 12.0 s\\ \hline
& 20 & 39.7 s & 1:23 m & 5.4 s s\\ \hline
\strut \ttt{A543} & 1 & --  &--  & 8:06 m \\ \hline
& 20 & 4:53 h & -- & 44.0 s \\ \hline
\strut \ttt{A553} & 20 &--  &--  &  1:22 h\\ \hline
\strut \ttt{lo6} & 1 &--  &--  & 3:19 h \\ \hline
& 20 &--  &--  & 27:11 m \\ \hline
\strut \ttt{5x5} & 1 & 45:39 m  & 11:52 m  &  1:25 m \\ \hline
& 20 & 5:16 m  & 5:18 m  & 18.5 s\\ \hline
\strut \ttt{cyclo60} & 1 &--  & 12:35 m  & 5:10 m \\ \hline
& 20 & 5:45 h  & 3:14 m  & 1:21 m\\ \hline
\end{tabular}
\vspace*{2ex} \caption{Triangulation}\label{data_tri}
\end{table}
\end{small}

For the computation of support hyperplanes the hybrid algorithm shows its power only for cones with truly large numbers of support hyperplanes, like \ttt{A553} or \ttt{cyclo60}. The third example \ttt{lo6} in Table \ref{data_sh} is a borderline case in which Pure Fourier-Motzkin elimination and the hybrid algorithm behave almost identically. The computation times of total pyramid decomposition are almost identical with those for triangulation since the only difference is that the simplicial cones must be stored.

\begin{small}
\begin{table}[hbt]
\centering
\begin{tabular}{|r|r|r|r|}\hline
\rule[-0.1ex]{0ex}{2.5ex}Input&threads&Fourier-Motzkin& hybrid\\ \hline
\strut \ttt{lo6} &1 & 39.3 s & 44.2 s \\ \hline
 &20 & 4.5 s & 4.1 s \\ \hline
\strut \ttt{cyclo60} &1 & -- & 2:52 m \\ \hline
&20  & 1:23 h & 44.3 s\\ \hline
\strut \ttt{A553} &1 & 2:48 h & 11:47 m \\ \hline
 &20 &  10:29 m& 1:08 m \\ \hline
\end{tabular}
\vspace*{2ex} \caption{Support hyperplanes}\label{data_sh}
\end{table}
\end{small}

\section{Evaluation of simplicial cones}\label{Eval}

The fast computation of triangulations via pyramid
decomposition must be accompanied by an efficient evaluation of
the simplicial cones in the triangulation $\Delta$, which, after the introduction of the pyramid decomposition, is almost
always the more time consuming step. Like the processing of pyramids, the evaluation of simplicial cones is parallelized in Normaliz.

Let $\sigma$ be a simplicial cone generated by the linearly
independent vectors $v_1,\dots,v_d$. The evaluation is based on
the \emph{generator matrix} $G_\sigma$ whose \emph{rows} are
$v_1,\dots,v_d$. Before we outline the evaluation procedure,
let us substantiate the remark made in Section \ref{LexTri}
that finding the support hyperplanes amounts to the inversion
of $G_\sigma$. Let $H_i$ be the support hyperplane of $\sigma$
opposite to $v_i$, given by the linear form
$\lambda_i=a_{1i}e_1^*+\dots+a_{di}e_d^*$ with coprime integer
coefficients $a_j$. Then
\begin{equation}
\lambda_i(v_k)=\sum_{j=1}^d v_{kj}a_{ji}=\begin{cases}
\hht_{H_i}(v_i),&k=i,\\ 0,&k\neq i.
\end{cases}\label{invert}
\end{equation}
Thus the matrix $(a_{ij})$ is $G_\sigma^{-1}$ up to scaling of
its columns. Usually the inverse is computed only for the first
simplicial cone in every pyramid since its support hyperplanes
are really needed. But matrix inversion is rather expensive,
and Normaliz goes to great pains to avoid it.

Normaliz computes sets of vectors, primarily Hilbert bases, but
also measures, for example the volumes of rational polytopes. A
polytope $P$ arises from a cone $C$ by cutting $C$ with a
hyperplane, and for Normaliz such hyperplanes are defined by
gradings: a \emph{grading} is a linear form $\deg:\ZZ^d\to\ZZ$
(extended naturally to $\RR^d$) with the following properties:
(i) $\deg(x)>0$ for all $x\in C$, $x\neq 0$, and (ii)
$\deg(\ZZ^d)=\ZZ$. The first condition guarantees that the
intersection $P=C\cap A_1$ for the affine hyperplane
$$
A_1=\{x\in\RR^d:\deg(x)=1\}
$$
is compact, and therefore a rational polytope. The second
condition is harmless for integral linear forms since it can be
achieved by extracting the greatest common divisor of the
coefficients of $\deg$ with respect to the dual basis.

The grading $\deg$ can be specified explicitly by the user or
chosen implicitly by Normaliz. The implicit choice makes only
sense if there is a natural grading, namely one under which the
extreme integral generators of $C$ all have the same degree.
(If it exists, it is of course uniquely determined.)

At present, Normaliz evaluates the simplicial cones $\sigma$ in
the triangulation of $C$ for the computation of the following
data:
\begin{enumerate}
\item[(HB)] the Hilbert basis of $C$,
\item[(LP)] the lattice points in the rational polytope
    $P=C\cap A_1$,
\item[(Vol)] the normalized volume $\vol(P)$ of the
    rational polytope $P$ (also called the
    \emph{multiplicity} of $C$),
\item[(HF)] the \emph{Hilbert} or \emph{Ehrhart function}
    $H(C,k)=|kP\cap \ZZ^d|$, $k\in\ZZ_+$.
\end{enumerate}

\subsection{Volume computation} Task (Vol) is the easiest, and Normaliz computes
$\vol(P)$ by summing the volumes $\vol(\sigma\cap A_1)$ where
$\sigma$ runs over the simplicial cones in the triangulation.
With the notation introduced above, one has
$$
\vol(\sigma\cap A_1)=\frac{|\det(G_\sigma)|}{\deg(v_1)\cdots\deg(v_d)}.
$$
For the justification of this formula note that the simplex
$\sigma\cap A_1$ is spanned by the vectors $v_i/\deg(v_i)$,
$i=1,\dots,d$, and that the vertex $0$ of the $d$-simplex
$\delta=\conv(0,\sigma\cap A_1)$ has (lattice) height $1$ over
the opposite facet $\sigma\cap A_1$ of $\delta$ so that
$\vol(\sigma\cap A_1)=\vol(\delta)$.

In pure volume computations Normaliz (since version 2.9)
utilizes the following proposition that often reduces the
number of determinant calculations significantly.

\begin{proposition}\label{uniexploit}
Let $\sigma$ and $\tau$ be simplicial cones sharing a facet $F$
Let $v_1,\dots,v_d$ span $\tau$ and let $v_d$ be opposite of
$F$. If $\det(G_\sigma)|=1$, then $|\det(G_\tau)|=\hht_F(v_d)$.
\end{proposition}

\begin{proof}
The proposition is a special case of \cite[Prop.~3.9]{BG}, but
is also easily seen directly. Suppose that $w_d$ is the
generator of $\sigma$ opposite to $F$. Then
$G_\sigma=\{v_1,\dots,v_{d-1},\allowbreak w_d\}$, and $|\det G_\sigma|=1$
by hypothesis. Therefore $v_1,\dots,v_{d-1},w_d$ span $\ZZ^d$.
With respect to this basis, the matrix of coordinates of
$v_1,\dots,v_d$ is lower trigonal with $1$ on the diagonal,
except in the lower right corner where we find $-\hht_F(v_d)$.
\end{proof}

Every new simplicial cone $\tau$ found by \textsc{ExtendTri} is
taken piggyback by an already known ``partner'' $\sigma$
sharing a facet $F$ with $\tau$. Therefore Normaliz records
$|\det G_\sigma|$ with $\sigma$, and if $|\det G_\sigma|=1$
there is no need to compute $|\det(G_\tau)|$ since the height
of the ``new'' generator $v_d$ over $F$ is known. Remark
\ref{timeRem}(b) contains some numerical data illuminating the
efficiency of this strategy that we call \emph{exploitation of
unimodularity}. One should note that it is inevitable to
compute $|\det(G_\sigma)|$ for the first simplicial cone in
every pyramid.

\subsection{Lattice points in the fundamental domain}
The sublattice $U_\sigma$ spanned by $v_1,\dots,v_d$ acts on $\RR^d$ by translation. The semi-open parallelotope
$$
\para(v_1,\dots,v_d)=\{q_1v_1+\dots+q_dv_d:0\le q_i<1\}.
$$
is a fundamental domain for this action; see Figure \ref{figfund}. In particular,
$$
E=E_\sigma=\para(v_1,\dots,v_d)\cap\ZZ^d
$$
is  a set of
representatives of the group $\ZZ^d/U_\sigma$.
The remaining tasks depend crucially on the set $E$.

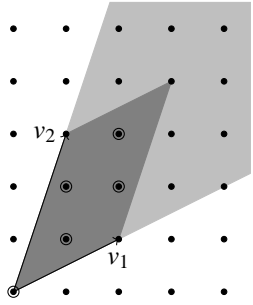
\begin{figure}[hbt]
\begin{center}
\begin{tikzpicture}[scale=0.7]
\filldraw[lightgray] (0,0) -- (1.833,5.5) -- (4.5,5.5) -- (4.5,2.25) -- cycle;
\filldraw[gray] (0,0) -- (1,3) -- (3,4) -- (2,1) -- cycle;
\draw[->] (0,0) -- (1,3) node at (0.6,3){$v_2$};
\draw[->] (0,0) -- (2,1) node at (2,0.6){$v_1$};
     \foreach \x in {0,...,4}
     \foreach \y in {0,...,5}
     {
     	\filldraw[fill=black]  (\x,\y)  circle (1.5pt);
     }
\draw (0,0)  circle (3.0pt);
\draw (1,1)  circle (3.0pt);
\draw (1,2)  circle (3.0pt);
\draw (2,2)  circle (3.0pt);
\draw (2,3)  circle (3.0pt);
\end{tikzpicture}
\end{center}
\caption{Lattice points in the fundamental domain }\label{figfund}
\end{figure}

For the efficiency of the evaluation it is important to
generate $E$ as fast as
possible.
One finds $E$ in two steps:
\begin{enumerate}
\item[(Rep)] find a representative of every residue class of the vectors in $\ZZ^d$,
    and
\item[(Mod)] reduce its coefficients with respect to the
    $\QQ$-basis $v_1,\dots,v_d$ modulo $1$.
\end{enumerate}

The first idea for (Rep) that comes to mind (and used in the
first version of Normaliz) is to decompose $\ZZ^d/U_\sigma$
into a direct sum of cyclic subgroups $\ZZ \overline u_i$,
$i=1,\dots,d$ where $u_1,\dots,u_d$ is a $\ZZ$-basis of $\ZZ^d$
and $\overline{\phantom{u}}$ denotes the residue class modulo
$U_\sigma$. The elementary divisor theorem guarantees the
existence of such a decomposition, and finding it amounts to a
diagonalization of $G_\sigma$ over $\ZZ$. But diagonalization
is even more  expensive than matrix inversion, and therefore it
is very helpful that a filtration of $\ZZ^d/U_\sigma$ with
cyclic quotients is sufficient. Such a filtration can be based
on trigonalization:

\begin{proposition}\label{trigon}
With the notation introduced, let $e_1,\dots,e_d$ denote the
unit vectors in $\ZZ^d$ and let $X\in\GL(d,\ZZ)$ such that
$XG_\sigma$ is an upper triangular matrix $D$ with diagonal
elements $a_1,\dots,a_d\ge 1$. Then the vectors
\begin{equation}
b_1e_1+\dots+b_de_d,\qquad 0\le b_i<a_i,\ i=1,\dots,d, \label{filter}
\end{equation}
represent the residue classes in $\ZZ^d/U_\sigma$.
\end{proposition}

\begin{proof}
Note that the rows of $XG_\sigma$ are a $\ZZ$-basis of
$U_\sigma$. Since $|\ZZ^d/U_\sigma|=|\det G_\sigma|=a_1\cdots
a_d$, it is enough to show that the elements listed represent
pairwise different residue classes. Let $p$ be the largest
index such that $a_p>1$. Note that $a_p$ is the order of the
cyclic group $\ZZ \overline e_p$, and that we obtain a
$\ZZ$-basis of $U_\sigma'=U_\sigma+\ZZ e_p$ if we replace the
$p$-th row of $XG_\sigma$ by $e_p$. If two vectors
$b_1e_1+\dots+b_pe_p$ and $b_1'e_1+\dots+b_p'e_p$ in our list
represent the same residue class modulo $U_\sigma$, then they are  even
more so modulo $U_\sigma'$. It follows that $b_i=b_i'$ for
$i=1,\dots,p-1$, and taking the difference of the two vectors,
we conclude that $b_p=b_p'$ as well.
\end{proof}

The first linear algebra step that comes up is therefore the
trigonalization
\begin{equation}
XG_\sigma=D.\label{trig}
\end{equation}

Let $G^\tr_\sigma$ be the transpose of $G_\sigma$. For (Mod) it
is essentially enough to reduce those $e_i$ modulo $1$ that
appear with a coefficient $>0$ in \eqref{filter}, and thus we
must solve the simultaneous linear systems
\begin{equation}
G^\tr_\sigma x_i=e_i,\qquad a_i>1, \label{repr}
\end{equation}
where we consider $x_i$ and $e_i$ as column vectors. In a crude
approach one would simply invert the matrix $G^\tr_\sigma$ (or
$G_\sigma$), but in general the number of $i$ such that $a_i>1$
is small compared to $d$ (especially if $d$ is large), and it
is much better to solve a linear system with the specific
multiple right hand side given by \eqref{repr}. The linear
algebra is of course done over $\ZZ$, using $a_1\cdots a_d$ as
a common denominator. Then Normaliz tries to produce the
residue classes and to reduce them modulo $1$ (or, over $\ZZ$,
modulo $a_1\cdots a_d$) as efficiently as possible.

For task (LP) one extracts the  vectors of degree $1$ from
$E$, and the degree $1$ vectors collected from all $\sigma$
from the set of lattice points in $P=C\cap A_1$. For (HB) one
first reduces the elements of $E\cup\{v_1,\dots,v_d\}$ to a
Hilbert basis of $\sigma$, collects these and then applies
``global'' reduction in $C$. This procedure has been described
in \cite{BI}.

\subsection{Hilbert series and Stanley decomposition}
The mathematically most
interesting task is (HF). The Hilbert series is defined by
$$
H_C(t)=\sum_{x\in C\cap\ZZ^d} t^{deg x}=\sum_{k=0}^\infty H(C,k)t^k,\qquad H(C,k)=|\{x\in C: \deg x =k\}|.
$$
It is well-known that $H_C(t)$ is the power series expansion of a rational function in $t$.

For a simplicial cone $\sigma$ spanned by $v_1,\dots,v_d$ as above one has
$$
H_\sigma(t)=\frac{h_0+h_1t+\dots+h_st^s}{(1-t^{g_1})\cdots (1-t^{g_d})},
\qquad g_i=\deg v_i,\qquad h_j=|\{x\in E_\sigma:\deg x=j\}|.
$$
This follows immediately from the disjoint decomposition
\begin{equation}
\sigma\cap\ZZ^d=\bigcup_{x\in E_\sigma} x+M_\sigma \label{disjoint}
\end{equation}
where $M_\sigma$ is the (free) monoid generated by
$v_1,\dots,v_d$.

However, one cannot compute $H_C(t)$ by simply summing these
functions over $\sigma\in\Delta$ since points in the intersections of the simplicial
cones $\sigma$ would be counted several times. Fortunately, the
intricate inclusion-exclusion problem can be avoided since
there exist \emph{disjoint} decompositions
\begin{equation}
C=\bigcup_{\sigma\in\Delta} \sigma\setminus S_\sigma\label{decosigma}
\end{equation}
of $C$ by semi-open
simplicial cones $\sigma\setminus S_\sigma$ where $S_\sigma$ is the union of some
facets (and not just arbitrary faces!) of $\sigma$. Following Kleinschmidt and Smilansky \cite{KZ} we call a decomposition of type \eqref{decosigma} a \emph{facet cover} of $\Delta$. (The name is motivated by the fact that each lower dimensional face of $\Delta$ is contained in exctly one of the ``surviving'' facets.)

Before we discuss the existence and computation of a facet cover, let us first derive a representation of the Hilbert series based on it. It generalizes the $h$-vector formula of McMullen-Walkup \cite[5.1.14]{BH}.

Let $\sigma\in\Delta$ and $x\in E_\sigma$, $x=\sum q_iv_i$. Then we
define $\epsilon(x)$ as the sum of all $v_i$ such that (i)
$q_i=0$ and (ii) the facet opposite to $v_i$ belongs to $S$. Since $(x+M_\sigma)\setminus
S=\epsilon(x)+x+M_\sigma$, we obtain the \emph{Stanley decomposition}
\begin{equation}
C\cap\ZZ^d =\bigcup_{\sigma\in\Delta} M_\sigma\setminus S_\sigma=\bigcup_{\sigma\in\Delta}\ \bigcup_{x\in E_\sigma} x+\epsilon(x)+M_\sigma.
\end{equation}
of $C\cap\ZZ^d$ into disjoint subsets. A Stanley decomposition into $4$ components is illustrated by Figure \ref{figdeco} in which lattice points in different components are marked differently.
\begin{figure}[hbt]
\begin{center}
\begin{small}
\begin{tikzpicture}
[scale=0.7,auto=left]
\draw (0,0) -- (8,6);
\draw (0,0) -- (8,0);
\draw (0,0) -- (8,8);
     \foreach \x in {0,...,8}
     \foreach \y in {0,...,\x}
     {
     	\draw (\x,\y)  circle (1.5pt);
     }
     \foreach \x/\y in { 0/0, 1/0, 2/0, 3/0, 4/0, 5/0, 6/0, 7/0, 8/0 }
     {
	     \filldraw[fill=black]  (\x,\y)  circle (1.5pt);	
     	
     }
     \foreach \x/\y in { 4/3, 5/3, 6/3,  7/3, 8/3}
     {
     	\filldraw[fill=black]  (\x,\y)  circle (1.5pt);     	
     }
     \filldraw[fill=black]  (8,6)  circle (1.5pt);
      \foreach \x/\y in { 2/1, 3/1, 4/1, 5/1, 6/1, 7/1, 8/1 }
      {
         	\draw  (\x,\y) node at (\x,\y) {$*$}  ;         	
      }
      \foreach \x/\y in { 6/4, 7/4, 8/4}
      {
      	\draw  (\x,\y) node at (\x,\y) {$*$}  ;	      	
      }
      \foreach \x/\y in { 3/2, 4/2, 5/2, 6/2, 7/2, 8/2}
      {
      	\draw  (\x,\y) node at (\x,\y) {$+$}  ;	      	
      }
      \draw  (7,5) node at (7,5) {$+$}  ;
      \draw  (8,5) node at (8,5) {$+$}  ;
\end{tikzpicture}
\end{small}
\end{center}
\caption{A Stanley decomposition}\label{figdeco}
\end{figure}
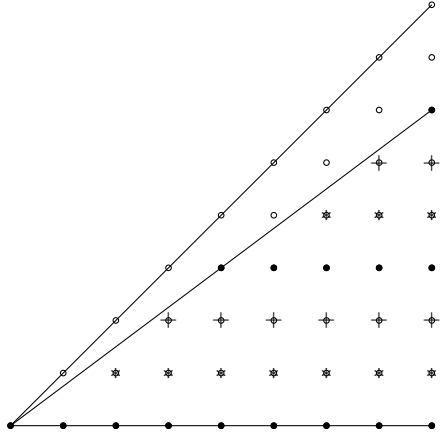

The series
$H_{\sigma\setminus S_\sigma}(t)$ is as easy to compute as
$H_\sigma(t)$:
\begin{align}
H_{\sigma\setminus S_\sigma}(t)&=\sum_{y\in M_\sigma\setminus S_\sigma} t^{\deg y} =\sum_{x\in E_\sigma}\sum_{z\in M_\sigma} t^{\deg x+\epsilon(x)+z}=\sum_{x\in E_\sigma} t^{\deg x+\epsilon(x)}H_\sigma(t) \nonumber\\ &=
\frac{\sum_{x\in E_\sigma}
t^{\deg \epsilon(x)+\deg x}}{(1-t^{g_1})\cdots (1-t^{g_d})}. \label{hilbS}
\end{align}
It only remains to sum the series $H_{\sigma\setminus S_\sigma}(t)$ over the triangulation $\Delta$.

The existence of a facet cover and (consequently) a Stanley decomposition of $C$ was shown by Stanley \cite[Theorem 5.2]{Sta} using
the existence of a line shelling of $C$ (proved by Bruggesser
and Mani). Instead of finding a shelling order for the
lexicographic triangulation (which is in principle possible),
Normaliz 2.0--2.5 used a line shelling for the decomposition,
as discussed in \cite{BI}.

This approach works well for cones
of moderate size, but has a major drawback: finding the sets
$S$ requires searching over the shelling order, and in
particular the whole triangulation must be stored.  We learned a much simpler principle for the
disjoint decomposition (already implemented in Normaliz 2.7) from Köppe and
Verdoolaege \cite{KV}. It was previously used by Kleinschmidt and Smilansky \cite{KZ} (also see Stanley \cite[p. 85]{Sta1}).
As a consequence, each simplicial cone in the triangulation can
be treated in complete independence from the others, and can
therefore be discarded once it has been evaluated (unless the
user insists on seeing the triangulation):

\begin{lemma}\label{inex}
Let $O_C$ be a vector in the interior of $C$ such that $O_C$ is
not contained in a support hyperplane of any simplicial
$\sigma$ in a triangulation of $C$. For $\sigma$ choose
$S_\sigma$ as the union of the support hyperplanes
$\cH^<(\sigma,O_C)$. Then the semi-open simplicial cones
$\sigma\setminus S_\sigma$ form a disjoint decomposition of
$C$.
\end{lemma}

See \cite{KV} for a proof. Figure \ref{figOrder} shows a facet cover resulting from Lemma \ref{inex}.
\begin{figure}[hbt]
\begin{center}
\begin{small}
\begin{tikzpicture}
[scale=0.7,auto=left, thick]
\foreach \x/\y in {0/2, 2/0, 5/0, 5/2, 5/4, 7/0, 7/4, 9/2}
\node [circle, draw, fill=black, inner sep=0pt, minimum width=2.5pt](n\x\y) at (\x,\y) {};
\node [circle, draw, fill=black, inner sep=0pt, minimum width=2.5pt](n23) at (2.5,3) {};
\node [circle, draw, inner sep=0pt, minimum width=3pt, label=above:$O_C$](OC) at (2.8,1.7) {};

\draw (n20) -- node[right=-2pt, pos=0.4]{$+$} node[left=-2pt, pos=0.4]{$-$} (n23);
\draw (n20) -- node[above=-2pt]{$+$} (n02);
\draw (n50) -- node[right=-2pt]{$-$} node[left=-2pt]{$+$} (n23);
\draw (n50) -- node[near end, right=-2pt]{$-$} node[near end, left=-2pt]{$+$} (n52);
\draw (n52) -- node[right=-2pt]{$-$} node[left=-2pt]{$+$} (n54);
\draw (n70) -- node[right=-2pt, pos=0.4]{$-$} node[left=-2pt, pos=0.4]{$+$} (n74);

\draw (n52) -- node[below=-2pt]{$+$} node[above=-2pt]{$-$} (n23);
\draw (n52) -- node[below=-2pt]{$-$} node[above]{$+$} (n74);
\draw (n50) -- node[right=-2pt]{$-$} node[left=-2pt]{$+$} (n74);

\draw (n02) -- node[below=-2pt]{$+$} (n23);
\draw (n23) -- node[right=5pt]{$+$} (n54);
\draw (n20) -- node[above=-2pt]{$+$} (n50);
\draw (n50) -- node[above=-2pt]{$+$} (n70);
\draw (n54) -- node[below=-2pt]{$+$} (n74);
\draw (n70) -- node[above=-2pt]{$+$} (n92);
\draw (n74) -- node[below=-2pt]{$+$} (n92);

\end{tikzpicture}
\end{small}
\end{center}
\caption{Using the order vector}\label{figOrder}
\end{figure}
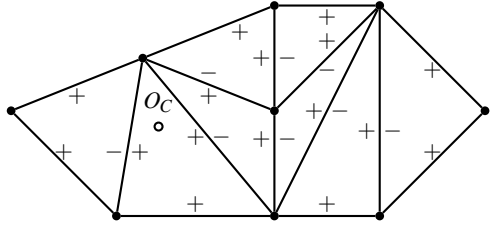

It is of course not possible to
choose an \emph{order vector} $O_C$ that avoids all hyperplanes
in advance, but this is not a real problem. Normaliz chooses
$O_C$ in the interior of the first simplicial cone, and works
with a lexicographic infinitesimal perturbation $O_C'$. (This
trick is known as "simulation of simplicity" in computational
geometry; see Edelsbrunner \cite{Ed}). If $O_C\in H^<$ (or
$O_C\in H^>$), then $O_C'\in H^<$ (or $O_C'\in H^>$). In the
critical case $O_C\in H$, we take the linear form $\lambda$
representing $H$ and look up its coordinates in the dual basis
$e_1^*,\dots,e_d^*$. If the first nonzero coordinate is
negative, then $O_C'\in H^<$, and else $O_C'\in H^>$.

At first it seems that one must compute the support hyperplanes
of $\sigma$ in order to apply Lemma \ref{inex}. However, it is
much better to solve the system
\begin{equation}
G^\tr_\sigma I^\sigma=O_C.\label{indicator}
\end{equation}
The solution $I^\sigma$ is called the \emph{indicator} of
$\sigma$. One has $O_C\in H^<$ (or $O_C\in H^>$) if
$I^\sigma_i<0$ (or $I^\sigma_i>0$) for the generator $v_i$
opposite to $H$ ($\lambda$ vanishes on $H$). Let us call
$\sigma$ \emph{generic} if all entries of $I^\sigma$ are
nonzero.

If $I^\sigma_i=0$---this happens rarely, and very rarely
for more than one index $i$---then we are forced to compute the
linear form representing the support hyperplane opposite of
$v_i$. In view of \eqref{invert} this amounts to solving the
systems
\begin{equation}
G_\sigma x=e_i, \qquad I^\sigma_i=0, \label{supp}
\end{equation}
simultaneously for the lexicographic decision.

If $\sigma$ is unimodular, in other words, if $|\det
G_\sigma|=1$, then the only system to be solved is
\eqref{indicator}, provided that $\sigma$ is generic. Normaliz
tries to take advantage of this fact by guessing whether
$\sigma$ is unimodular, testing two necessary conditions:
\begin{itemize}
\item[(PU1)] Every $\sigma$ (except the first) is inserted
    into the triangulation with a certain generator $x_i$.
    Let $H$ be the facet of $\sigma$ opposite to $x_i$. If
    $\hht_H(x_i)>1$, then $\sigma$ is nonunimodular. (The
    number $\hht_H(x_i)$ has been computed in the course of
    the triangulation.)
\item[(PU2)] If $\gcd(\deg v_1,\dots,\deg v_d)>1$, then
    $\sigma$ is not unimodular.
\end{itemize}
If $\sigma$ passes both tests, we call it \emph{potentially
unimodular}. (Data on the efficiency of this test will be given
in Remark \ref{timeRem}(a)).

After these preparations we can describe the order in which
Normaliz treats the trigonalization \eqref{trig} and the linear
systems \eqref{repr}, \eqref{indicator} and \eqref{supp}:
\begin{enumerate}
\item[(L1)] If $\sigma$ is potentially unimodular, then
    \eqref{indicator} is solved first. It can now be
    decided whether $\sigma$ is indeed unimodular.
\item[(L2)] If $\sigma$ is not unimodular, then the
    trigonalization \eqref{trig} is carried out next. In
    the potentially unimodular, but nongeneric case, the
    trigonalization is part of the solution of \eqref{supp}
    (with multiple right hand side).
\item[(L3)] In the nonunimodular case, we now solve the
    system \eqref{repr} (with multiple right hand side).
\item[(L4)] If $\sigma$ is not potentially unimodular and
    not generic, it remains to solve the system
    \eqref{supp} (with multiple right hand side).
\end{enumerate}

As the reader may check, it is never necessary to perform all
$4$ steps. In the unimodular case, (L1) must be done, and
additionally (L2) if $\sigma$ is nongeneric. If $\sigma$ is not
even potentially unimodular, (L2) and (L3) must be done, and
additionally (L4) if it is nongeneric. In the potentially
unimodular, but nonunimodular case, (L1), (L2) and (L3) must be
carried out.

\subsection{Presentation of Hilbert series}
We conclude this section with a brief discussion of
the computation and the representation of the Hilbert series by
Normaliz. The reader can find the necessary background in
\cite[Chapter 6]{BG}.

Summing the Hilbert series \eqref{hilbS} is very simple if they
all have the same denominator, for example in the case in which
the generators of $C$ (or at least the extreme integral
generators) have degree $1$. For efficiency, Normaliz first
forms ``denominator classes'' in which the Hilbert series with
the same denominator are accumulated. At the end, the class
sums are added over a common denominator that is extended
whenever necessary. This yields a ``raw'' form of the Hilbert
series of type
\begin{equation}
H_C(t)=\frac{R(t)}{(1-t^{s_1})\cdots(1-t^{s_r})},\qquad R(t)\in \ZZ[t],\label{raw}
\end{equation}
whose denominator in general has $>d$ factors.

In order to find a presentation with $d$ factors, Normaliz
proceeds as follows. First it reduces the fraction to lowest
terms by factoring the denominator of \eqref{raw} into a
product of cyclotomic polynomials:
\begin{equation}
H_C(t)=\frac{Z(t)}{\zeta_{z_1}\cdots\zeta_{z_w}},\qquad Z(t)\in \ZZ[t],
\quad \zeta_{z_j}\nmid Z(t),\label{cycl}
\end{equation}
which is of course the most economical way for representing
$H_C(t)$ (as a single fraction).  The orders and the
multiplicities of the cyclotomic polynomials can easily be
bounded since all denominators in \eqref{hilbS} divide
$(1-t^\ell)^d$ where $\ell$ is the least common multiple of the
degrees $\deg x_i$. So we can find a representation
\begin{equation}
H_C(t)=\frac{F(t)}{(1-t^{e_1})\cdots (1-t^{e_d})},\qquad F(t)\in \ZZ[t],\label{period}
\end{equation}
in which $e_d$ is the least common multiple of the orders  of
the cyclotomic polynomials that appear in \eqref{cycl},
$e_{d-1}$ is the least common multiple of the orders that have
multiplicity $\ge 2$ etc. Normaliz produces the presentation
\eqref{period} whenever the degree of the numerator remains of
reasonable size.

It is well-known that the Hilbert function itself is a
quasipolynomial:
\begin{equation}
H(C,k)=q_0(k)+q_1(k)k+\dots+q_{d-1}(k)k^{d-1},\qquad k\ge0,
\end{equation}
where the coefficients $q_j(k)\in\QQ$ are periodic functions of
$k$ whose common period is the least common multiple of the
orders of the cyclotomic polynomials in the denominator of
\eqref{cycl}. Normaliz computes the quasipolynomial, with the
proviso that its period is not too large. It is not hard to see
that the periods of the individual coefficients are related to
the representation \eqref{period} in the following way: $e_k$
is the common period of the coefficients
$q_{d-1},\dots,q_{d-k}$. The leading coefficient $q_{d-1}$ is
actually constant (hence $e_1=1$), and related to the
multiplicity by the equation
\begin{equation}
q_{d-1}=\frac{\vol(P)}{(d-1)!}.\label{highest}
\end{equation}
Since $q_{d-1}$ and $\vol(P)$ are computed completely
independently from each other, equation \eqref{highest} can be
regarded as a test of correctness for both numbers.

The choice \eqref{period} for $H_C(t)$ is motivated by the
desire to find a standardized representation whose denominator
conveys useful information. The reader should note that this
form is not always the expected one. For example, for
$C=\RR_+^2$ with $\deg(e_1)=2$ and $\deg(e_2)=3$, the three
representations \eqref{raw}--\eqref{period} are
$$
\frac{1}{(1-t^2)(1-t^3)}=\frac{1}{\zeta_1^2\zeta_2\zeta_3}=\frac{1-t+t^2}{(1-t)(1-t^6)}.
$$
Actually, it is unclear what the most natural standardized
representation of the Hilbert series as a fraction of two
polynomials should look like, unless the denominator is
$(1-t)^d$. Perhaps the most satisfactory representation should
use a denominator $(1-t^{p_1})\cdots(1-t^{p_d})$ in which the
exponents $p_i$ are the degrees of a homogeneous system of
parameters (for the monoid algebra $K[\ZZ^d\cap C]$ over an
infinite field $K$). At present Normaliz cannot find such a
representation (except the one with the trivial denominator
$(1-t^\ell)^d)$), but future versions may contain this functionality.

\section{Computational results}\label{Comp}

In this section we want to document that the algorithmic
approach described in the previous sections (and \cite{BI}) is
very efficient and masters computations that appeared
inaccessible some years ago. We compare Normaliz 3.0 to 4ti2, version 1.6.6 \cite{4ti2},
for
Hilbert basis computations and to LattE integrale, version 1.7.3  \cite{LatInt}, for
Hilbert series.

Almost all computations were run on a Dell PowerEdge R910 with 4 Intel
Xeon E7540 (a total of 24 cores running at 2 GHz), 128 GB of
RAM and a hard disk of 500 GB. The remaining computations were run on a SUN
xFire 4450 with a comparable configuration. In parallelized computations we
have limited the number of threads used to $20$. As the large
examples below show, the parallelization scales efficiently. In
Tables \ref{times1} and \ref{times2}
serial execution is indicated by \ttt{1x} whereas \ttt{20x} indicates parallel
execution with a maximum of $20$ threads. Normaliz needs
relatively little memory. Almost all  Normaliz computations
mentioned run stably with $< 1$ GB of RAM.

Normaliz is distributed as open source under the GPL. In
addition to the source code, the distribution contains
executables for the major platforms Linux, Mac and Windows.

\subsection{Overview of the examples}

We have chosen the following test candidates:

\begin{enumerate}
\item \ttt{CondPar}, \ttt{CEffPl} and \ttt{PlVsCut} come
    from social choice theory. \ttt{CondPar}
    represents the Condorcet paradox, \ttt{CEffPl}
    computes the Condorcet efficiency of plurality voting,
    and \ttt{PlVsCut} compares plurality voting to cutoff,
    all for $4$ candidates. See Schürmann \cite{Sch} for
    more details.
\item \ttt{4x4}, \ttt{5x5} and \ttt{6x6} represent monoids
    of ``magic squares'': squares of size $4\times 4$,
    $5\times 5$ and $6\times 6$ to be filled with
    nonnegative integers in such a way that all rows,
    columns and the two diagonals sum to the same ``magic
    constant''. They belong to the standard LattE
    distribution \cite{LatInt}.
\item \ttt{bo5} and \ttt{lo6} belong to the area of
    statistical ranking; see Sturmfels and Welker
    \cite{SW}. \ttt{bo5} represents the boolean model for
    the symmetric group $S_5$ and \ttt{lo6} represents the
    linear order model for $S_6$.
\item \ttt{small} and \ttt{big} are test examples used in
    the development of Normaliz without further importance.
    \ttt{small} has already been discussed in \cite{BI}.
\item \ttt{cyclo36}, \ttt{cyclo38}, \ttt{cyclo42} and
    \ttt{cyclo60} represent the cyclotomic monoids of
    orders $36$, $38$, $42$ and $60$. They are additively
    generated by the pairs $(\zeta,1)\in \CC\times \ZZ_+$
    where $\zeta$ runs over the roots of unity of the given
    order. They have been discussed by Beck and Ho\c{s}ten
    \cite{BeH}.
\item \ttt{A443} and \ttt{A553} represent monoids defined by
    dimension $2$ marginal distributions of dimension $3$
    contingency tables of sizes $4\times4\times3$ and
    $5\times5\times3$. They had been open cases in the
    classification of Ohsugi and Hibi \cite{OH} and were
    finished in \cite{BHIKS}.
\item \ttt{cross10}, \ttt{cross15} and \ttt{cross20} are
    (the monoids defined by) the cross polytopes of
    dimensions $10$, $15$ and $20$ contained in the LattE
    distribution \cite{LatInt}.
\end{enumerate}

\begin{table}[hbt]
\begin{small}
\centering
\rlap{
\begin{tabular}{|r|r|r|r|r|r|r|r|}\hline
\rule[-0.1ex]{0ex}{2.5ex}Input& edim & rank & $\#$ext & $\#$supp& $\#$Hilb & $\#$ triangulation & $\#$ Stanley dec\\ \hline
\strut \ttt{CondPar} & 24 & 24 & 234 & 27 & 242 & 1,344,671 & 1,816,323\\ \hline
\strut \ttt{PlVsCut} & 24 & 24 & 1,872 & 28 & 9,621 & 257,744,341,008 & 2,282,604,742,033\\ \hline
\strut \ttt{CEffPl}  & 24 & 24 & 3,928 & 30 & 25,192 & 347,225,775,338 & 4,111,428,313,448\\ \hline
\strut \ttt{4x4}     & 16 & 8 & 20 & 16 & 20 & 48 & 48\\ \hline
\strut \ttt{5x5}     & 25 & 15 & 1,940 & 25 & 4,828 & 14,615,011 & 21,210,526\\ \hline
\strut \ttt{6x6}     & 36 & 24 & 97,548 & 36 & 522,347 & -- & --\\ \hline
\strut \ttt{bo5}     & 31 & 27 & 120 & 235 & 120 & 20,853,141,970 & 20,853,141,970\\ \hline
\strut \ttt{lo6}     & 16 & 16 & 720 & 910 & 720 & 5,796,124,824 & 5,801,113,080\\ \hline
\strut \ttt{small}   & 6 & 6 & 190 & 32 & 34,591 & 4580 & 2,276,921\\ \hline
\strut \ttt{big}     & 7 & 7 & 27 & 56 & 73,551 & 542 & 18,788,796\\ \hline
\strut \ttt{cyclo36} & 13 & 13 & 36 & 46,656 & 37 & 44,608 & 46,656\\ \hline
\strut \ttt{cyclo38} & 19 & 19 & 38 & 923,780 & 39 & 370,710 & 923,780\\ \hline
\strut \ttt{cyclo42} & 13 & 13 & 42 & 24,360 & 43 & 153,174 & 183,120\\ \hline
\strut \ttt{cyclo60} & 17 & 17 & 60 & 656,100 & 61 & 11,741,300 & 13,616,100\\ \hline
\strut \ttt{A443}    & 40 & 30 & 48 & 4,948 & 48 & 2,654,272 & 2,654,320\\ \hline
\strut \ttt{A553}    & 55 & 43 & 75 & 306,955 & 75 & 9,248,466,183 & 9,249,511,725\\ \hline
\strut \ttt{cross10} & 11 & 11 & 20 & 1,024 & 21 & 512& 1,024\\ \hline
\strut \ttt{cross15} & 16 & 16 & 30 & 32,678 & 31 & 16,384 & 32,768\\ \hline
\strut \ttt{cross20} & 21 & 21 & 40 & 1,048,576 & 41 & 524,288 & 1,048,576\\ \hline
\end{tabular}}
\end{small}
\vspace*{2ex} \caption{Numerical data of test
examples}\label{data}
\end{table}

The columns of Table \ref{data} contain the values of
characteristic numerical data of the test examples $M$, namely:
edim is the embedding dimension, i.~e., the rank of the lattice
in which $M$ is embedded by its definition, whereas rank is the
rank of $M$. $\#$ext is the number of the extreme rays of the
cone $\RR_+M$, and $\#$supp the number of its support
hyperplanes. $\#$Hilb is the size of the Hilbert basis of
$M$.

The last two columns list the number of simplicial cones in the
triangulation and the number of components of the Stanley
decomposition. These data are not invariants of $M$. However,
if the triangulation uses only lattice points of a lattice
polytope $P$ (all examples starting from \ttt{bo5}), then the
number of components of the Stanley decomposition is exactly
the normalized volume of $P$.

The open entries for \ttt{6x6} seem to be out of reach
presently. The Hilbert series of \ttt{6x6} is certainly a
challenge for the future development of Normaliz. Other
challenges are \ttt{lo7}, the linear order polytope for $S_7$
and the first case of the cyclotomic monoids \ttt{cyclo105}
that is not covered by the theorems of Beck and Ho\c{s}ten
\cite{BeH}. Whether \ttt{cyclo105} will ever become computable,
is quite unclear in view of its gigantic number of support
hyperplanes. However, we are rather optimistic for \ttt{lo7};
the normality of the linear order polytope for $S_7$ is an open
question.

\subsection{Hilbert bases}

Table \ref{times1} contains the computation times for the
Hilbert bases of the test candidates. When comparing 4ti2 and
Normaliz one should note that 4ti2 is not made for the input of
cones by generators, but for the input via support hyperplanes
(\ttt{CondPar} -- \ttt{6x6}). The same applies to the Normaliz
dual mode \ttt{-d}. While Normaliz is somewhat faster even in
serial execution, the times are of similar magnitude. It is
certainly an advantage that its execution   has been
parallelized. When one runs Normaliz with the primary algorithm
on such examples it first computes the extreme rays of the cone
and uses them as generators.

Despite of the fact that several examples could not be expected
to be computable with 4ti2, we tried. We stopped the
computations when the time had exceeded 150 h (T) or the memory
usage had exceeded 100 GB (R). However, one should note that
\ttt{A553} (and related examples) can be computed by ``LattE
for tea, too'' (\url{http://www.latte-4ti2.de}), albeit with a
very large computation time; see \cite{BHIKS}. This approach
uses symmetries to reduce the amount of computations.
\begin{table}[hbt]
\centering
\begin{small}
\begin{tabular}{|r|r|r|r|r|r|}\hline
\rule[-0.1ex]{0ex}{2.5ex}Input&\ttt{4ti2}&\ttt{Nmz -d 1x}&\ttt{Nmz -d 20x}&\ttt{Nmz  -N 1x}&\ttt{Nmz -N 20x}\\ \hline
\strut \ttt{CondPar} & 0.024 s           & 0.014 s       & 0.026 s        & 2.546 s        & 0.600 s\\ \hline
\strut \ttt{PlVsCut} & 6.672 s           & 0.820 s       & 0.476 s        & --             & --\\ \hline
\strut \ttt{CEffPl}  & 6:08 m            & 28.488 s      & 3.092 s        & --             & --\\ \hline
\strut \ttt{4x4}     & 0.008 s           & 0.003 s       & 0.011 s        & 0.005 s        & 0.016 s\\ \hline
\strut \ttt{5x5}     & 3.823 s           & 1.004 s       & 0.339 s        & 1:06 m         & 23.714 s\\ \hline
\strut \ttt{6x6}     & 115:26:31 h       & 14:19:39 h    & 1:19:34 h      & --             & --\\ \hline
\strut \ttt{bo5}     & T                 & --            & --             & 0.273 s        & 0.174 s\\ \hline
\strut \ttt{lo6}     & 31:09 m           & 1:46 m        & 39.824 s       & 1:08 m         & 13:369 s\\ \hline
\strut \ttt{small}   & 48:19 m           & 18:45 m       & 3:25 m         & 1.935 s        & 1.878 s\\ \hline
\strut \ttt{big}     & T                 & --            & --             & 1:45 m         & 15.636 s\\ \hline
\strut \ttt{cyclo36} & T                 & --            & --             & 0.774 s        & 0.837 s\\ \hline
\strut \ttt{cyclo38} & R                 & --            & --             & 6:32:50 h    & 1:04:04 h\\ \hline
\strut \ttt{cyclo60} & R                 & --            & --             & 2:55 m         & 1:02 m\\ \hline
\strut \ttt{A443}    & T                 & --            & --             & 1.015 s        & 0.270 s\\ \hline
\strut \ttt{A553}    & R                 & --            & --             & 44:11 m        & 4:24 m\\ \hline
\end{tabular}
\vspace*{2ex} \caption{Computation times for Hilbert
bases}\label{times1}
\end{small}
\end{table}

In Table \ref{times1} the option \ttt{-d} indicates the dual algorithm, and \ttt{-N} indicates the the primal algorithm for Hilbert bases. The number $n$ of threads is given by $n$\ttt{x}.

The examples \ttt{CEffPl}, \ttt{PlVsCut}, \ttt{5x5} and
\ttt{6x6} are clear cases for the dual algorithm. However, it
is sometimes difficult to decide whether the primary,
triangulation based algorithm or the dual algorithm is faster. As
\ttt{small} clearly shows, the dual algorithm behaves badly if
the final Hilbert basis
is large, even if the number of support hyperplanes is small.

The computation time of \ttt{bo5} which is close to zero is
quite surprising at first glance, but it has a simple
explanation: the lexicographic triangulation defined by the
generators in the input file is unimodular so that all pyramids
have height $1$, and the partial triangulation is empty.

The computation time for the Hilbert basis of \ttt{cyclo38} is
large compared to the time for the Hilbert series in Table
\ref{times2}. The reason is the large number of support
hyperplanes together with a large number of candidates for the
Hilbert basis. Therefore the reduction needs much time.

The Hilbert basis computations in the Normaliz primary mode
show the efficiency of partial triangulations (see Section
\ref{partial}). Some numerical data are contained in
\cite{BHIKS}.

We have omitted the \ttt{cross} examples from the Hilbert basis
computation in view of the obvious unimodular triangulation of
the cross polytopes (different from the one used by Normaliz).
\ttt{cross20} needs 16 s for \ttt{Nmz -N x1}.

\subsection{Hilbert series} Now we compare the computation
times for Hilbert series of Normaliz and LattE. One should note
that the computations with LattE are not completely done by open
source software: for the computation
of Hilbert series it invokes the commercial program Maple.
LattE has a variant for the computation of Hilbert polynomials
that avoids Maple; however, it can only be applied to lattice
polytopes (and not to rational polytopes in general).

There are three columns with computation times for LattE. The
first, \ttt{LattE ES}, lists the times for LattE alone, without
Maple, the second, \ttt{LattE + M ES}, the combined computation
time of LattE and Maple (both for Hilbert series), and the third, \ttt{LattE EP}, the
computation time of LattE for the Hilbert polynomial. In all of
these three columns we have chosen the best time that we have
been able to reach with various parameter settings for LattE.
However, LattE has failed on many candidates, partly because it
produces enormous output files. We have stopped it when the
time exceeded 150 hours (T), the memory usage was more than 100
GB RAM (R) or it has  produced more than 400 GB of output (O).
These limitation were imposed by the system available for
testing. In three cases it has exceeded the system stack limit;
this is marked by~S.

It is easy to see that \ttt{cross}$n$ has Hilbert series $(1+t)^n/(1-t)^{n+1}$. Therefore it is a good test candidate for the correctness of the algorithm.

\begin{table}[hbt]
\centering
\begin{small}
\begin{tabular}{|r|r|r|r|r|r|}\hline
\rule[-0.1ex]{0ex}{2.5ex}Input& \ttt{LattE ES} & \ttt{LattE+M ES} & \ttt{LattE EP}  &\ttt{Nmz  1x} &\ttt{Nmz 20x}\\ \hline
\strut \ttt{CondPar} & O        & S          & --        & 18.085 s   & 8.949 s\\ \hline
\strut \ttt{PlVsCut} & O        & S          & --        & --         & 145:43:03 h\\ \hline
\strut \ttt{CEffPl}  & O        & S          & --        & --         & 197:45:10 h\\ \hline
\strut \ttt{4x4}     & 0.329 s  & 4.152 s    & --        & 0.006 s    & 0.018 s\\ \hline
\strut \ttt{5x5}     & O        & 72:39:23 h & --        & 3:59 m     & 1:12 m\\ \hline
\strut \ttt{bo5}     & T        & T          & T         & 82:40:18 h & 6:41:12 h \\ \hline
\strut \ttt{lo6}     & R        & R          & T         & 13:02:44 h & 1:21:52 h \\ \hline
\strut \ttt{small}   & 46.266 s & 30:15 m    & 22.849 s  & 0.233 s    & 0.095 s\\ \hline
\strut \ttt{big}     & R        & R          & 10.246 s  & 1.473 s    & 0.148 s\\ \hline
\strut \ttt{cyclo36} & R        & R          & 23:03 m   & 1.142 s    & 1.106 s\\ \hline
\strut \ttt{cyclo38} & R        & R          & R         & 26.442 s   & 22.789 s\\ \hline
\strut \ttt{cyclo42} & R        & R          & 1:44:07 h & 3.942 s    & 1.521 s\\ \hline
\strut \ttt{cyclo60} & R        & R          & T         & 5:57 m     & 1:44 m\\ \hline
\strut \ttt{A443}    & R        & R          & R         & 49.541 s   & 18.519 s\\ \hline
\strut \ttt{A553}    & R        & R          & T         & 88:21:18 h & 6:29:05 h\\ \hline
\strut \ttt{cross10} & T        & T          & 9.550 s   & 0.016 s    & 0.022 s \\ \hline
\strut \ttt{cross15} & R        & R          & 21:48 m   & 0.536 s    & 0.533 s \\ \hline
\strut \ttt{cross20} & R        & R          & R         & 26.678 s   & 26.029 s\\ \hline

\end{tabular}
\vspace*{2ex} \caption{Computation times for Hilbert series and Hilbert polynomials}\label{times2}
\end{small}
\end{table}

\begin{remark}\label{timeRem}
(a) From the Hilbert series calculation of \ttt{PlVsCut} we
have obtained the following statistics on the types of
simplicial cones:
\begin{enumerate}
\item $61,845,707,957$ are unimodular,
\item $108,915,272,879$ are not unimodular, but satisfy
    condition (PU1), and of these
\item $62,602,898,779$ are potentially unimodular.
\end{enumerate}
This shows that condition (PU2) that was added at a
later stage has a satisfactory effect. (The number of
potentially unimodular, but nonunimodular simplicial
cones is rather high in this class.) The average value
of $|\det G_\sigma|$ is $\approx 10$. This can be read
off Table \ref{data} since the sum of the $|\det
G_\sigma|$ is the number of components of the Stanley
decomposition.

The number of nongeneric simplicial cones is $129,661,342$. The
total number $s$ of linear systems that had to be solved for
the computation of the Hilbert series is bounded by
$516,245,872,838\le s \le 516,375,534,180$.

The total number of pyramids was $80,510,681$. It depends on
the number of parallel threads that are allowed.

(b) For examples with a high proportion of unimodular cones the
exploitation of unimodularity based on Proposition
\ref{uniexploit} is very efficient in volume computations. With this strategy, \ttt{lo6} requires only $102,526,351$  determinant
calculations instead of $5,801,113,080$. For \ttt{PlVsCut} it
saves about $25\%$.

(c) For the examples from social choice theory
(\ttt{CondPar}, \ttt{CEffPl}, \ttt{PlVsCut}) Schürmann
\cite{Sch} has suggested a very efficient improvement via
symmetrization that replaces the Ehrhart series of a polytope
by the generalized Ehrhart series of a projection. Normaliz now
has an offspring, NmzIntegrate, that computes generalized
Ehrhart series; see Bruns and Söger \cite{BS}.

The volumes of the pertaining polytopes had already been
computed by Schürmann with LattE integrale. This information
was very useful for checking the correctness of Normaliz.

(d) The short Normaliz computation times for the \ttt{cyclo}
and \ttt{cross} examples are made possible by the special
treatment of simplicial facets in the Fourier-Motzkin
elimination; see \cite{BI}.
\end{remark}

\section{Acknowledgement}

The authors like to thank  Mihai Cipu, Matthias K\"{o}ppe,
Achill Sch\"{u}rmann, Bernd Sturmfels, Alin \c{S}tefan and
Volkmar Welker  for the test examples that were used during the
recent development of Normaliz and for their useful comments.
We are grateful to Elisa Fascio for her careful reading of the
first version and Lukas Katthän for a reference to \cite{Sta1} in connection with the order vector.

Bogdan Ichim was partially supported a grant of CNCS - UEFISCDI, project number PN-II-RU-TE-2012-3-0161
during the preparation of this work and the development of Normaliz.

The Normaliz project is supported by the DFG Schwerpunktprogramm 1489 ``Algorithmische und ex\-pe\-rimentelle Methoden in Algebra, Geometrie und Zahlentheorie''

\end{document}